\documentclass[a4paper]{article}

\usepackage[utf8]{inputenc}
\usepackage{todonotes} 
\newcommand{\todofloat}[1]{}

\usepackage{graphicx}
\graphicspath{{images/}}
\usepackage{adjustbox}
\usepackage{siunitx}
\usepackage{booktabs}
\usepackage{multirow}
\usepackage{comment}
\usepackage{enumitem}
\usepackage{url}
\usepackage{authblk}

\usepackage{amsmath, amsthm, amssymb, amsfonts, epsfig, xspace, mathtools}
\usepackage{stmaryrd}

\newcommand{\vx}{\mathbf{x}}
\newcommand{\vy}{\mathbf{y}}
\newcommand{\vp}{\mathbf{p}}
\newcommand{\tD}{\mathbf{D}}
\newcommand{\Id}{\mathbf{I}}
\newcommand{\Sym}[1]{\ensuremath{\mathrm{Sym}^{#1\times #1}}}

\newcommand{\R}{\mathbb{R}}
\newcommand{\N}{\mathbb{N}}
\newcommand{\E}{\ensuremath{\mathcal{E}}}
\newcommand{\Vm}{\ensuremath{V_{\textrm{m}}}}
\newcommand{\norm}[1]{\left\lVert#1\right\rVert}
\newcommand{\diff}[1]{\ensuremath{\operatorname{d}\!{#1}}}

\newcommand{\abs}[1]{\left\lvert#1\right\rvert}
\newcommand{\sphere}{\mathbb{S}}

\DeclareMathOperator{\proj}{proj}

\DeclareMathOperator*{\argmin}{arg\,min}

\usepackage{algorithm2e}

\usepackage{pgfplots}
\pgfplotsset{compat=newest}
\usepackage{tikz}
\usetikzlibrary{calc}
\usetikzlibrary{spy}
\usetikzlibrary{decorations.pathreplacing}
\usetikzlibrary{shapes}
\usetikzlibrary{through}
\usetikzlibrary{intersections}
\usetikzlibrary{decorations.markings}
\usetikzlibrary{3d}
\usetikzlibrary{fit}
\usetikzlibrary{arrows.meta}
\usetikzlibrary{arrows}
\usetikzlibrary{external}

\newtheorem{theorem}{Theorem}
\newtheorem{lemma}{Lemma}
\newtheorem{proposition}{Proposition}
\newtheorem{corollary}{Corollary}
\newtheorem{remark}{Remark}

\usepackage{xspace}
\newcommand{\TheMethod}{GEASI\xspace}
\usepackage{xr-hyper}
\usepackage{hyperref}
\usepackage{cleveref}

\usepackage{quoting,xparse}

\NewDocumentCommand{\bywhom}{m}{
  {\nobreak\hfill\penalty50\hskip1em\null\nobreak
   \hfill\mbox{\normalfont(#1)}%
   \parfillskip=0pt \finalhyphendemerits=0 \par}%
}

\NewDocumentEnvironment{pquotation}{m}
  {\begin{quoting}[
     indentfirst=true,
     leftmargin=\parindent+.5cm,
     rightmargin=\parindent+.5cm]\itshape}
  {\bywhom{#1}\end{quoting}}

\newcommand{\mytextcirc}[1]{\raisebox{.5pt}{\textcircled{\raisebox{-.9pt}{\textbf{#1}}}}}

\title{\TheMethod: Geodesic-based Earliest Activation Sites Identification in cardiac models}

\author[1,4]{Thomas Grandits}
\author[1,5,6]{Alexander Effland}
\author[1,4]{Thomas Pock}
\author[2]{Rolf Krause}
\author[3,4]{Gernot Plank}
\author[2,*]{Simone Pezzuto}

\affil[1]{\small Institute of Computer Graphics and Vision, TU Graz, Inffeldgasse 16, 8010 Graz, Austria}
\affil[2]{Center for Computational Medicine in Cardiology, Euler Institute, Universit\`a della Svizzera italiana, via la Santa 1, 6900 Lugano, Switzerland}
\affil[3]{Gottfried Schatz Research Center - Division of Biophysics, Medical University of Graz, Neue Stiftingtalstra\ss e 6/IV, 8010 Graz, Austria}
\affil[4]{BioTechMed-Graz, Graz, Austria}
\affil[5]{Silicon Austria Labs (TU Graz SAL DES Lab), Graz, Austria}
\affil[6]{Institute for Applied Mathematics, University of Bonn, Endenicher Allee 60, 53115 Bonn, Germany}
\affil[*]{Corresponding author: \texttt{simone.pezzuto@usi.ch}}

\date{}

\begin{document}

\maketitle

\begin{abstract}
The identification of the initial ventricular activation sequence
is a critical step for the correct personalization of patient-specific
cardiac models.  In healthy conditions, the Purkinje network is the main
source of the electrical activation, but under pathological conditions
the so-called earliest activation sites (EASs) are possibly sparser and more localized.
Yet, their number, location and timing may not be easily inferred
from remote recordings, such as the epicardial activation or the
12-lead electrocardiogram (ECG), due to the underlying complexity of the model.

In this work, we introduce \TheMethod (\textbf{G}eodesic-based \textbf{E}arliest \textbf{A}ctivation \textbf{S}ites \textbf{I}dentification) as a novel approach to simultaneously identify all EASs. 
To this end, we start from the anisotropic eikonal equation modeling cardiac electrical activation and exploit its Hamilton--Jacobi formulation to minimize a given objective function, e.g.~the quadratic mismatch to given activation measurements.
This versatile approach can be extended to estimate the number of activation sites by means of the topological gradient, or fitting a given ECG.

We conducted various experiments in 2D and 3D for in-silico models and an in-vivo intracardiac recording collected from a patient undergoing cardiac resynchronization therapy.
The results demonstrate the clinical applicability of \TheMethod for potential future personalized models and clinical intervention.


\end{abstract}

\section{Introduction}
\label{sec:intro}
In this work, we address the central question of identifying earliest
activation sites (EASs) in a propagation model for ventricular activation.
In a healthy (human) heart, ventricles are activated via a specific pathway
that originates in the atrio-ventricular node, continues in the His bundle and
the Purkinje network, to eventually spread in the myocardium through
Purkinje-myocardial junctions~\cite{vigmond2016modeling}.  These junction
points can be effectively modeled by a discrete set of EASs, that form the
initial condition of the propagation model.  Unfortunately, the precise
structure of the set of EASs (defined by their number, location and timing) cannot
be detected \emph{in vivo}, and rule-based approaches are limited by
inter-patient variability.  More importantly, severe pathological conditions
such as intraventricular conduction disorders are directly associated with 
partially malfunctioning activation pathways, hence corresponding to a
pathological set of EASs.  A correct and possibly automatic identification of
EASs from non-invasive or minimally-invasive recordings is therefore of high
clinical relevance, especially in selecting the optimal treatment for the
patient~\cite{auricchio2003clinical}.

A particularly suitable propagation model in the context of EASs is the
anisotropic eikonal equation, which was originally exploited as a convenient
approximation of the monodomain and bidomain models~\cite{keener1991eikonal,
franzone1990wavefront}, but is nowadays more often utilized for its
computational efficiency~\cite{neic_efficient_2017,pezzuto_evaluation_2017}.
This work, however, leverages the eikonal model from a novel perspective, based
on Hamilton-Jacobi formalism and geodesics~\cite{crandall1984some}, to enable a gradient-based
approach for localizing the EASs termed \TheMethod (\textbf{G}eodesic-based \textbf{E}arliest \textbf{A}ctivation \textbf{S}ites \textbf{I}dentifi\-cation).
In detail, we start from the anisotropic eikonal equation as a common model for cardiac electrophysiology~\cite{franzone_spreading_1993}, in which the EASs define boundary conditions at specific sites. 
For numerical reasons, the eikonal equation is solved using the Fast Iterative Method (FIM)~\cite{FIM_tetra}.
The main goal of our approach is the minimization of a given objective functional depending on the solution of the anisotropic eikonal equation as a function of the EASs.
Here, a feasible optimization strategy involves the Hamilton--Jacobi formalism, which promotes a tractable derivative with respect to the EASs~\cite{crandall1984some}.
Note that this derivative is geometrically related to the tangent of the geodesic at the EASs.
In this respect, a geodesic connects an EAS such as a Purkinje entry point to an observation through a path of minimum distance in a predefined metric.
Finally, we exploit the aforementioned methods to introduce \TheMethod, which in its core employs a quadratic mismatch between the eikonal solution and the measurements, such as observations of an electro-anatomical mapping, in the objective function.
In summary, \TheMethod can therefore fit the parameters of the eikonal model to clinical recordings in a very efficient and flexible manner.



We emphasize that \TheMethod is not limited to this quadratic objective function and can straightforwardly be extended to other scenarios.
In this work, we additionally investigate two such extensions: the topological gradient designed for estimating the number of EASs and fitting of EASs of an eikonal-based ECG model to a clinically recorded ECG.

Changing the number of EASs and the effect of this action on the objective function is evaluated by means of the topological gradient. 
The concept of topological gradients can be readily introduced via the Hamilton--Jacobi theory.
Here, we consider the splitting of a single EAS into a pair of two EASs symmetrically arranged at infinitesimal distance along a given direction in a dipole-like fashion.
Thus, the topological gradient provides a criterion to decide whether an EAS should be split or not.
In particular, this approach promotes a simple model as a starting point with too little complexity to represent the measurement data and increase the number of source sites until the encountered activations are properly approximated.

In combination with the pseudo-bidomain model, a template-based action potential and the lead field theory, the eikonal model also results in an almost-real-time ECG simulator~\cite{neic_efficient_2017,pezzuto_evaluation_2017} with remarkable physiological accuracy~\cite{franzone_spreading_1993}.
Here, we extend a previously proposed approach~\cite{Pezzuto2021ECG}, based on this ECG model, to solve the inverse ECG problem, i.e.~we exploit \TheMethod to localize EASs purely from ECG data.
Numerical experiments in \Cref{sec:experiments} demonstrate that the proposed approach is capable of finding the optimal EASs even in high-fidelity cardiac models.
We visually summarized \TheMethod and its applications in \Cref{fig:graphical_abstract}.

\begin{figure}[tb]
    \includegraphics[width=.99\linewidth, trim={0cm 0cm 0cm 0cm}, clip]{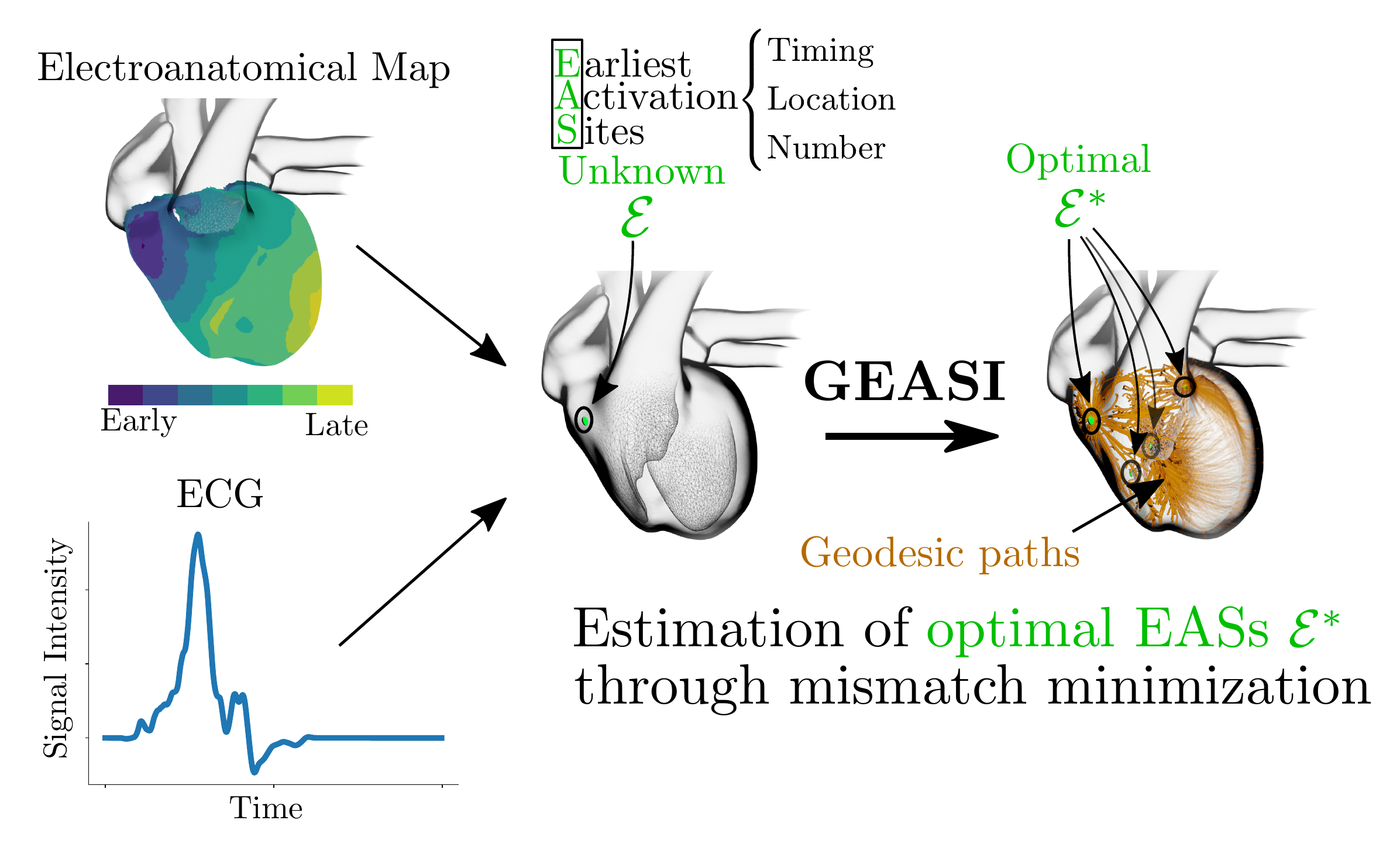}
    \caption{
        \TheMethod in a nutshell: 
        The EASs --- given by their timings, locations and number --- are important parameters that define cardiac activation.
        \TheMethod is able to identify these parameters by exploiting the Hamilton--Jacobi formulation.
        This formulation allows for an efficient optimization scheme for minimizing mismatches to either activation maps or resulting quantities such as the ECG.
        Using the topological gradient we can additionally estimate the number of EASs.
    }
    \label{fig:graphical_abstract}
\end{figure}

\subsection{Related Work}
In what follows, we briefly review similar and related approaches to \TheMethod.

From a physiological perspective, EASs can be derived from an automatically
generated Purkinje network~\cite{costabal2016generating}, closely following the
actual anatomy of the heart. The approach, anatomically-tailored but not
patient-specific, is indicated in the case of a generic healthy activation and
a (complete or partial) bundle branch block.  When dense endocardial mapping
data are available, the Purkinje network can be estimated
automatically~\cite{palamara_effective_2015, barber2018automatic,
ulysses_optimization-based_2018}.  In the method proposed by
Palamara \emph{et al.}, the Purkinje network is created from
intra-cardiac measurements by dividing the endocardium into regions of
influence for each Purkinje entry point.  For this purpose, an isotropic
eikonal equation is solved for each measurement point to compute the regions of
influence in a Voronoi-diagram like fashion.  According to fractal laws,
Purkinje entry points are subsequently either moved, deleted or generated to
better fit the observed activation on the endocardium.  The connection with
\TheMethod becomes apparent once we consider the underlying problem in terms of
geodesics and regions of influence, further discussed in \Cref{sec:methods}.
We can observe that \TheMethod is a generalization of the aforementioned
approach, since it gives rise to Voronoi partitions with non-linear boundaries.
Thus, \TheMethod can be applied to heterogeneous conduction velocities and
heterogeneous fiber directions.

An alternative formulation of the Purkinje network is based on a very sparse
set of EASs embedded in thin, fast-conducting layers in both ventricular
endocardia~\cite{Lee2019}.  In some sense, this approach can be referred to as
\emph{lumped Purkinje network} formulation.  Here, the number of EASs is
drastically reduced.  In fact, a few sites per chamber are generally sufficient
to correctly capture the activation and reproduce the surface
ECG~\cite{Potse2014}.  Overall, the parameters of the eikonal model with lumped
Purkinje network are just the location, the number and the activation onset of
the EASs, as well as the conduction velocities in the thin layers and the
myocardium.  The problem of fitting these parameters to clinical data has
already been considered in the literature for the conduction
velocities~\cite{barone2020experimental, Grandits2020PIEMAP,
stella2020integration,grandits_inverse_2020}.  However, the optimization of
EASs received limited attention so far with only a few works dealing with
activation onsets~\cite{kallhovd2018inverse} or
locations~\cite{kunisch_inverse_2019}.  The simultaneous optimization of EASs
(especially their number) and conduction velocity has been analyzed only very
recently in~\cite{Giffard-Roisin2019, grandits_inverse_2020,Pezzuto2021ECG}.

In the work by Kunisch \emph{et al.}, the authors recast the problem
of localizing EASs as a shape optimization problem.  In their viscous eikonal
formulation, the EASs are modelled as small spherical holes in the domain whose
boundaries impose the activation onset.  Under sufficient smoothness
assumptions, an adjoint state can be defined through the shape derivative with
respect to these boundaries, and therefore be exploited to optimize the EASs.
This approach is efficient and can be applied to multiple pacing sites,
although---in contrast to our approach---no topological changes are permitted
by the formulation, impeding both a change in the number of EASs and movement
from the interior to the boundary of the domain.  Additionally, the coalescence
of multiple sites needs special treatment in the shape derivative (not
addressed in \cite{kunisch_inverse_2019}).  The viscous eikonal formulation,
moreover, introduces a curvature-dependent conduction velocity, potentially
strong at EASs, posing limitations on the radius of the spherical holes.

It is worth noting that all previous works, either based on the
full~\cite{palamara_effective_2015} or a
lumped~\cite{kunisch_inverse_2019,kallhovd2018inverse} Purkinje network,
require local measurements of activation times, e.g.~endocardial maps, whereas
\TheMethod can be applied to fit epicardial recordings and even the surface
ECG.  This aspect is relevant in view of non-invasive personalization of
patient-specific models, as recently advocated~\cite{corral2020digital}.
Moreover, a major challenge in cardiac personalization addressed by \TheMethod
is the estimation of the ground truth number of EASs.  A possible solution is
to consider a large number of EASs densely covering the earliest activation
region, and successively removing sites according to some predefined
rule~\cite{Pezzuto2021ECG}.  For instance, an optimization procedure could
determine the optimal activation onset of all sites, and then remove those with
a very late onset.  In a previous study, we optimized the initiation times
along with the anisotropic conductivity tensors by manually deriving
the Fast Iterative Method~\cite{FIM_tetra}.  The large number of EASs results
in a highly ill-posed inverse problem and consequently requires further
regularization to successively remove initiation sites.  While providing good
results with respect to the measured activation
times~\cite{grandits_inverse_2020, Pezzuto2021ECG}, this approach heavily
relies on initial choice of EASs and the selected regularization strategy.

\subsection{Notation}
\label{sec:notation}
We denote the identity matrix by~$\Id$ and the space of symmetric and positive definite matrices in~$\R^d$ by $\Sym{d}$.
Throughout this work, $B_\zeta(\vx)$ and $\overline{B}_\zeta(\vx)$ refer to the open and closed ball with radius~$\zeta>0$ around~$\vx$.
We denote by $\mathring{X}$ the interior of the set $X$.
$[A]_{ij}$ denotes the entry in the $i$-th row and $j$-th column of a matrix $A \in \R^{n_1 \times n_2}$.
Furthermore, $A\succ B$ if $A-B$ is positive definite for $A,B\in\Sym{d}$, and we set $\llbracket N\rrbracket\coloneqq\{1,\ldots,N\}$.
The set of unit vectors in~$\R^d$ is denoted by~$\sphere^{d-1}$.
The Dirac measure of a set~$S$ is referred to as~$\delta_{[S]}$.
Further, we use the notation~$C^0(X,Y)$ for the space of continuous functions mapping from~$X$ to~$Y$ endowed with the norm~$\Vert\cdot\Vert_{C^0(X,Y)}$, and we denote by $C^k(X,Y)$ the associated space of $k$-times continuously differentiable functions equipped with the norm~$\Vert\cdot\Vert_{C^k(X,Y)}$.
We use the symbol~$C^{k,\alpha}(X,Y)$ for the H\"older space with exponent~$\alpha$ and norm~${\Vert\cdot\Vert_{C^{k,\alpha}(X,Y)}}$.
Finally, we denote by~$L^p(X,Y)$ and~$W^{m,p}(X,Y)$ the $p$-Lebesgue space and the Sobolev space of $m$-times weakly differentiable and $p$-integrable functions, and we set $H^m(X,Y)=W^{m,2}(X,Y)$.

\subsection{Structure of the Work}
In \Cref{sec:methods}, we successively introduce the eikonal equation and the objective functional, which are the buildings blocks of \TheMethod.
Based on the introduced algorithm, we present in \Cref{sec:methods_extension} the topological gradient as well as the ECG fitting problem.
Then, we elaborate in \Cref{sec:discretization} on efficient discretization schemes and implementation detail of the proposed method.
In \Cref{sec:experiments}, we consider the problem of estimating the initiation sites for different models---primarily in-silico experiments, but also one in-vivo experiment.
Further aspects of future work are addressed in \Cref{sec:discussion}.

\section{\TheMethod}
\label{sec:methods}
Next, we introduce the \TheMethod method, which encompasses the following ingredients. 
In \Cref{sec:intro_example}, we review the anisotropic eikonal equation and its associated Hamilton--Jacobi formulation.
Subsequently, in \Cref{sec:intro_min_problem} we analyze a general objective function involving the solution of the anisotropic eikonal equation from a functional-analytical perspective.
\Cref{sec:geodesic_extraction} deals with the gradient computation of the distance function, which is later exploited in the aforementioned objective functional.
Finally, all introduced concepts are combined in \Cref{sec:geodesic_opt} to define \TheMethod.

\subsection{Eikonal equation}
\label{sec:intro_example}

We consider the computational domain~$\Omega\subset\R^d$ for $d\geq 2$, which in most cases represents the myocardium.
Further, let $\E$ be the subset of $N$ pairs $\{(\vx_i,t_i)_{i=1}^N\}\in\mathcal{U}_N\coloneqq\Omega^N\times (T_\mathrm{min},T_\mathrm{max})^N$ for a priori given $T_\mathrm{min}<T_\mathrm{max}$ and fixed $N$.
Throughout this work, $\E$ is a set of EASs, where $N$ is the number of EASs, $\vx_i$ and $t_i$ are the location and timing of the $i$-th site, respectively.
Let $\phi_\E:\Omega\to\R$ be the unique solution of the anisotropic eikonal equation with prescribed values on~$\E$, which is commonly referred to as the activation map.
That is, $\phi_\E(\vx)$ is the first arrival time at $\vx\in\Omega$
of the propagating action potential.
Hence, $\phi_\E$ solves
\begin{equation}\label{eq:eikonal}
\left\{
\begin{array}{rcll}
\sqrt{\tD(\vx)\nabla\phi_\E(\vx)\cdot\nabla\phi_\E(\vx)} &=& 1,
& \vx\in\Omega\setminus\{\vx_1,\ldots,\vx_N\}, \\[1em]
\phi_\E(\vx_i) &=& t_i, & (\vx_i,t_i)\in\E,
\end{array}
\right.
\end{equation}
where $\tD\in C^1(\overline\Omega,\Sym{d})$ describes the anisotropic conduction.
In the model, the anisotropy arises from the fiber alignment inside the heart~\cite{franzone_spreading_1993}.
Recall that $\Sym{d}$ is defined as the set of positive definite and symmetric $d\times d$-matrices, which gives rise to the definition of the norm $\Vert\mathbf{p}\Vert_{\tD}\coloneqq\sqrt{\tD\mathbf{p}\cdot\mathbf{p}}$ for $\mathbf{p}\in\R^d$.
Note that the assumptions already guarantee that
\begin{equation*}
\lambda_{*} \mathbf{I}\prec \tD(\vx) \prec \lambda^{*}\mathbf{I}
\end{equation*}
for all $\vx\in\overline\Omega$ and finite bounds $0 < \lambda_{*} \le \lambda^{*} < \infty$.
It is well known that the eikonal equation~\eqref{eq:eikonal} admits a unique viscosity solution according to the theory of Hamilton--Jacobi equations~\cite{crandall1984some}. 
The Lipschitz continuous solution of the eikonal equation $\phi_\E\in C^{0,1}(\overline\Omega)$ is of the form
\begin{equation}\label{eq:mineiko}
\phi_\E(\vx) = \min_{(\vy,t)\in\E}
\{ t + \delta(\vx,\vy) \}
\end{equation}
where $\delta(\vx,\vy)$ denotes the \emph{geodesic distance}
\begin{equation}\label{eq:distance}
\delta(\vx,\vy)
= \inf_{\widehat\gamma\in H^1([0,1],\overline\Omega)} \left\{
L(\widehat{\gamma}):\widehat\gamma(0)=\vx,\widehat\gamma(1)=\vy
\right\}
\end{equation}
given the length functional
\begin{equation}
L(\gamma)\coloneqq\int_0^1 \Vert\dot{\gamma}(t)\Vert_{\tD^{-1}\left(\gamma(t)\right)}\diff{t}.
\label{eq:lengthFunctional}    
\end{equation}
Thus, the induced Riemannian metric for two vectors $\mathbf{v},\mathbf{w}\in\mathbb{R}^d$ is 
\begin{equation}
\langle\mathbf{v},\mathbf{w}\rangle_{\gamma(t)}\coloneqq\tD^{-1}\left(\gamma(t)\right)\mathbf{v}\cdot\mathbf{w}.
\label{eq:RiemannianMetric}
\end{equation}
We note that the infimum~$\gamma$ in~\eqref{eq:distance} is actually attained, and by the geodesic equation we can even deduce $\gamma\in C^{0,1}([0,1],\overline\Omega)$ (see e.g.~\cite{bornemann2006finite}).
Indeed, in the definition~\eqref{eq:distance}, we first note that $\Vert\vp \Vert_{\tD^{-1}(\vx)}\leq\lambda_\ast^{-1}\Vert\vp\Vert_2$ for all $\vp\in\R^d$ and $\vx\in\Omega$.
Then, for any segment $[\vx,\vy]$ fully contained in $\overline\Omega$ we have that $\delta(\vx,\vy)\leq\lambda_\ast^{-1}\Vert\vx-\vy\Vert_2$ since the segment is a geodesic path in the Euclidean norm. 
\Cref{fig:geodesics} illustrates a single geodesic path in red on a domain with a continuously varying conduction velocity and isotropic conduction.
\begin{figure}[tb]
\resizebox{\linewidth}{!}{
\includegraphics{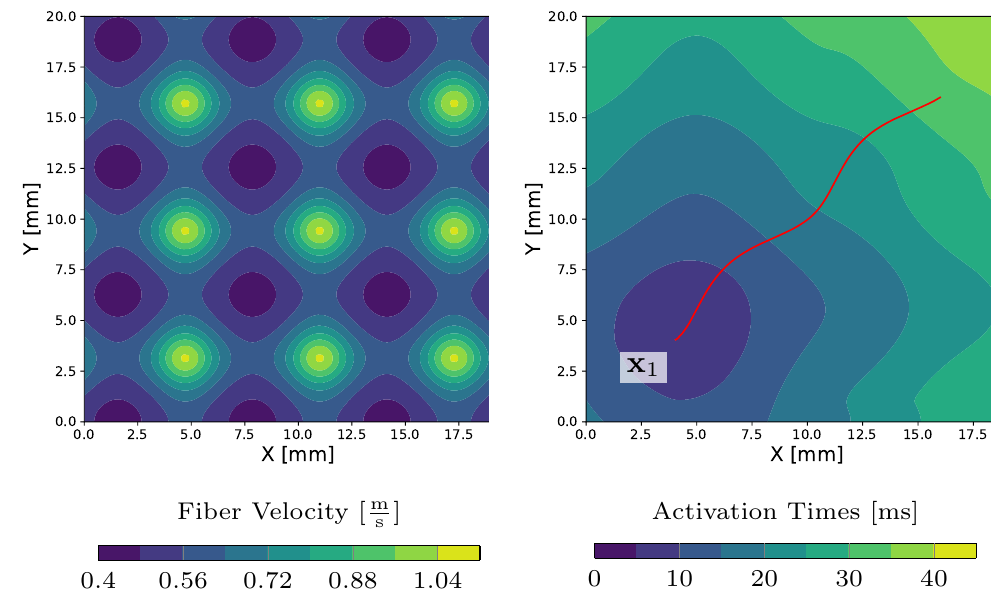}
}
\caption{
Left: fixed velocity field~$c$ in~$\Omega$.
Right: contour plot of the associated anisotropic eikonal equation with anisotropic conduction~$\tD(\vx)=c(\vx)^2\Id$ along with the geodesic path joining the EAS $\vx_1$ with an arbitrary point.}
\label{fig:geodesics}
\end{figure}

\label{rev:R1compat}
When $N>1$, all pairs $(\vx_i,t_i),(\vx_j,t_j)\in\E$ must satisfy the subsequent compatibility condition 
\begin{equation}
t_i-t_j \leq\delta(\vx_i,\vx_j)
\label{eq:compatibilityCondition}
\end{equation}
in order to ensure the existence of a solution. This fact is a direct consequence
of~\eqref{eq:mineiko}, because non-compatible data can not exist w.r.t.~the eikonal equation.  
For the purpose of this work, this
condition is not too restrictive, since we aim at identifying EASs rather than
enforcing them.
Interestingly, the condition is also physiologically sound: if a stimulus at some location $\vx_j$ is applied too late, e.g., right after the passage of an activation front originating from $\vx_i$, it should not trigger another propagation, because the tissue is already depolarized. In fact, under such circumstances the activation time $t_j$ at $\vx_j$ would be larger than
the travel time from $\vx_i$, that is $t_i + \delta(\vx_i,\vx_j)$, clearly violating~\eqref{eq:compatibilityCondition}.

The Hamilton--Jacobi formulation is essential for computing perturbations of~$\E$, which is conducted in the following subsection.

\subsection{Objective functional}
\label{sec:intro_min_problem}
The overall objective of this work is the minimization of a given functional $\mathcal{J}:C^{0,1}(\overline\Omega)\to\R$ depending on the activation map $\phi_\E$ with respect to~$\E$, i.e.
\begin{equation}
m_N\coloneqq\min_{\E\in\overline{\mathcal{U}_N}}
\mathcal{J}(\phi_\E).
\label{eq:initial_opt_problem}
\end{equation}
For instance, the objective could describe the minimization of a mismatch (in the least-squares sense) between the simulated activation and the activation detected from epicardial, as well as endocardial mapping (see \Cref{sec:experiments}).
The objective functional can also involve the activation map implicitly: 
In \Cref{sec:experiments_ecg}, we utilize the mismatch between the recorded and simulated 12-lead surface ECG as a metric for optimization.

In what follows, we prove the existence of minimizers for~\eqref{eq:initial_opt_problem} for varying~$N$.
To this end, we define for $\E=\{(\vx_i,t_i)_{i=1}^N\}$ 
\begin{equation}
\Phi_N(\vx,\vx_1,\ldots,\vx_N,t_1,\ldots,t_N)\coloneqq\phi_\E(\vx).
\label{eq:PhiNDefinition}    
\end{equation}
\begin{lemma}
$\Phi_N\in C^{0,1}(\overline{\Omega\times\mathcal{U}_N})$ is a bounded function of its arguments. \end{lemma}
\begin{proof}
Using~\eqref{eq:mineiko}, we immediately see that
\[
\Phi_N(\vx,\vx_1,\ldots,\vx_N,t_1,\ldots,t_N)=\min_{i=1,\ldots,N}\left\{ t_i+\delta(\vx_i,\vx)\right\}.
\]
The Lipschitz continuity of~$\delta$ as well as the compactness of $\overline{\Omega\times\mathcal{U}_N}$ imply the statement.
\end{proof}
We note that Rademacher's theorem ensures the differentiability of~$\Phi_N$ almost everywhere. 
Non-differentiability with respect to $\vx$ occurs for instance at $\vx=\vx_i$, but also in the presence of front collisions. 
An immediate consequence of this lemma is the following
\begin{theorem}[Existence]
\label{thm:existence}
If $\mathcal{J}$ is uniformly continuous, then the problem~\eqref{eq:initial_opt_problem} admits at least one minimum.
\end{theorem}
\begin{proof}
The previous lemma and the uniform continuity of~$\mathcal{J}$ imply the existence of at least one minimum.
\end{proof}
\begin{proposition}
\label{prop:decreasing}
Under the hypotheses of \Cref{thm:existence}, $m_N$ is a non-increasing function of $N$. 
Moreover, if there exists $N$ such that $m_{N+1}=m_{N}$, then $m_{N+n}=m_{N}$ for all $n\geq 1$.
\end{proposition}
\begin{proof}
The first claim immediately follows from the definition of~$m_N$ and set inclusion arguments.
To prove the second claim, we assume that
$m_{N+1}=m_N$ for some $N$ and $m_{N+2} < m_{N+1}$.
However, the choice $\vx_{N+1}=\vx_{N+2}$ and $t_{N+1}=t_{N+2}$ results in a contradiction.
\end{proof}
\begin{corollary}
If $N$ is bounded from above by $N_{\max}$, then $\min_{N\leq N_{\max}}m_N$ has at least one minimum.
\end{corollary}
\begin{remark}
    \label{rmk:bounded}
\begin{enumerate}
\item 
From a practical point of view, this corollary ensures that by adding new EASs, we either improve the objective function or we keep the same level of accuracy.  
This is also seen in the experiments in \Cref{sec:experiments}, where coalescence of two or more sites is observed if introducing too many EASs.
\item
The minimum in~\eqref{eq:initial_opt_problem} is in general not unique as it depends on the choice of $\mathcal{J}$ and on the order of the EASs.
In principle, by permuting EASs we obtain the same value of the minimum.
In particular, this symmetry induces a periodic partition of the set $\mathcal{U}_N$.
Each partition is associated with a specific choice of the order of the EASs.
From a numerical point of view, this may constitute a problem for methods based on random sampling. For deterministic steepest descent algorithms, the problem is mitigated by the fact that we rarely cross the boundary between two partitions, e.g., by swapping points, unless the two points coincide. 
\item
In general, we cannot take $N$ unbounded with no further hypotheses on~$\mathcal{J}$.
Suppose for instance that $\mathcal{J}$ is minimized by $\phi(\vx)=c$ for some constant $c\in\R$.
Then, $\inf_{N\in\N}m_N$ attains no minimum.
Indeed, we cannot represent a constant function with \eqref{eq:mineiko} if $\E$ is only countable.
However, we can approximate the constant with arbitrary precision with a sufficiently large number~$N$ of EASs.
\end{enumerate}
\end{remark}

\subsection{Exponential Map}
\label{sec:geodesic_extraction}
In what follows, we compute the Riemannian exponential map to derive an expression for the variation of the distance function.
In particular, we discuss the relation of the derivatives of $\Phi_N$ and the geodesic path.

We briefly recall fundamental concepts in Riemannian geometry.
Given $\vx\in\Omega$ and a tangent vector~$\mathbf{v}\in\mathcal{V}$ for a sufficiently small neighborhood $\mathcal{V}$ around the origin of the tangent space at~$\vx$,
the exponential map $\operatorname{Exp}_{\vx}:\mathcal{V}\to \operatorname{Exp}_\vx (\mathcal{V}) \subset \Omega$ is given by $\operatorname{Exp}_{\vx}(\mathbf{v})=\gamma(1)$,
where $\gamma\in C^{0,1}([0,1],\Omega)$ is a geodesic path with $\mathbf{v}\coloneqq\dot{\gamma}(0)$ and $\operatorname{Exp}_\vx (\mathcal{V})$ is the corresponding image domain.
The logarithmic map $\operatorname{Log}_\vx:\Omega \supset \operatorname{Exp}_\vx (\mathcal{V}) \to\mathcal{V}$ is the inverse of the exponential map $\operatorname{Exp}_{\vx}^{-1}$.
In other words, the logarithmic map of $\vy\in\Omega$ identifies the tangent vector $\dot{\gamma}(0)$ of a geodesic path~$\gamma$ emanating from $\vx$ and ending at $\vy$.
\begin{proposition}[Variation of the distance function]\label{prop:deriv_dist}
Let $\vx,\vy\in\mathcal{X}$, where $\mathcal{X}\subset\mathring{\Omega}$ is sufficiently small such that all points inside are connected by unique geodesics.
Then the variation of $\delta(\vx,\vy)$ with respect to $\vy$ with $\mathbf{w}=\operatorname{Log}_{\vx}(\vy)$ reads as
\begin{equation}
\nabla_{\vx}\delta(\vx,\vy)=-\frac{\tD^{-1}(\vx)\mathbf{w}}{\Vert\mathbf{w}\Vert_{\tD^{-1}(\vx)}}.
\label{eq:source_gradient_wo_residuals}
\end{equation}
\end{proposition}
\begin{proof}
Suppose that $\gamma$ is a geodesic with respect to the Riemannian metric in~\eqref{eq:RiemannianMetric} realizing the distance $\delta(\vy,\vx)$, i.e.~$\gamma(0)=\vy$, $\gamma(1)=\vx$ and
\[
\delta(\vx,\vy)=\int_0^1\Vert\dot{\gamma}(t)\Vert_{\tD^{-1}(\gamma(t))}\diff{t}.
\]
Let $\widetilde\gamma:[0,1]\times(-R,R)\to\mathring\Omega$ for small $R>0$ be a smooth variation of $\gamma$ such that $\widetilde\gamma(t,0)=\gamma(t)$ for all $t\in[0,1]$.
The first variation formula~\cite[Chapter~10]{ONeill} with $c=\Vert\dot{\gamma}(t)\Vert_{\tD^{-1}(\gamma(t))}$ for $t\in[0,1]$ implies
\begin{align}
\nabla_{\vx} \delta(\vx,\vy)(\widetilde\gamma)=\frac{1}{c}\Bigg(&-\int_0^1
\langle\ddot{\gamma}(t),\partial_2\widetilde\gamma(t,0)\rangle_{\gamma(t)}\diff{t}
-\sum_{i=1}^k\langle\Delta\dot{\gamma}(t_i),\partial_2\widetilde\gamma(t_i,0)\rangle_{\gamma(t_i)}\notag\\
&+\langle\dot\gamma(1),\partial_2\widetilde\gamma(1,0)\rangle_{\gamma(1)}-\langle\dot\gamma(0),\partial_2\widetilde\gamma(0,0)\rangle_{\gamma(0)}\Bigg).
\label{eq:FVF}
\end{align}
Here, $0<t_1<\cdots<t_k<1$ are possible discontinuities of the geodesic curve and
$\Delta \dot{\gamma}(t_i)=\dot{\gamma}(t_i^+)-\dot{\gamma}(t_i^-)$, where $\dot{\gamma}(t_i^-)$ and $\dot{\gamma}(t_i^+)$ denote the one-sided derivatives from the left and the right, respectively.
The derivative of $\widetilde\gamma$ with respect to the second argument is denoted by~$\partial_2\widetilde\gamma$.
Since $\gamma$ is assumed to be geodesic and smooth, the first two summands in~\eqref{eq:FVF} vanish.

By adjusting $\widetilde\gamma$ such that $\partial_2\widetilde\gamma(1,0)=0$ and observing that $\dot\gamma(0)=\operatorname{Log}_{\vx}(\vy)$ we have proven 
\[
\nabla_{\vx}\delta(\vx,\vy)(\widetilde\gamma)=-\frac{\tD^{-1}(\gamma(0))\dot{\gamma}(0)}{\Vert\dot{\gamma}(0)\Vert_{\tD^{-1}(\gamma(0))}}\cdot\partial_2\widetilde\gamma(0,0),
\]
which readily implies \eqref{eq:source_gradient_wo_residuals}.
\end{proof}
In \Cref{prop:deriv_dist}, we assumed uniqueness and smoothness of the geodesic curve, which is in general not ensured.
In practice, the influence of geodesics violating these assumptions is negligible and thus in \TheMethod only consider~\eqref{eq:source_gradient_wo_residuals} for all computations.

As before, let $\phi_\E$ be the solution of the eikonal equation with given $\E=\{(\vx_i,t_i)_{i=1}^N\}$.
This admits a natural definition of region of influences as 
\begin{equation}
    \mathcal{R}_i\coloneqq\{\vx\in\overline\Omega:\phi_\E(\vx)=t_i+\delta(\vx,\vx_i)\}.
    \label{eq:region_of_influence}
\end{equation}
Note that each point in the interior of $\mathcal{R}_i$ is thus assigned to a single EAS through means of the geodesic distance.
Furthermore, the derivatives of $\Phi_N$ with respect to $\vx_i$ and $t_i$ at $\vx\in\mathring{\mathcal{R}}_i$ read as
\begin{align*}
\nabla_{\vx_i}\Phi_N(\vx,\vx_1,\ldots,\vx_N,t_1,\ldots,t_N)&=\nabla_{\vx_i}\delta(\vx_i,\vx),\\
\partial_{t_i}\Phi_N(\vx,\vx_1,\ldots,\vx_N,t_1,\ldots,t_N)&=1,
\end{align*}
where we note that function is not differentiable on the boundary of the regions of interest.
To compute the exponential map, we solve for each $i=1,\ldots,N$ the following initial value problem
\begin{equation}
\begin{cases}
\dot{\gamma_i}(t)=-\tD(\gamma_i(t))\nabla\phi_{\{(\vx_i,t_i)\}}(\gamma_i(t)),\\
\gamma_i(0)=\vx
\end{cases}
\label{eq:ode_eikonal_geodesics}
\end{equation}
for $\vx\in\Omega$.
The regularity and boundedness of~$\tD$ and~$\phi_{\{(\vx_i,t_i)\}}$ already imply the existence of solutions.
Then, we define the piecewise geodesic path $\gamma$ as $\gamma(t)=\gamma_i(t)$ if $\gamma(t)\in\mathcal{R}_i$.
Furthermore, 
\begin{equation}
\overline t=\argmin_{t>0}\{\gamma(t)\in\overline{B}_\zeta(\vx_i)\text{ for }i=1,\ldots,N\}
\label{eq:ZetaT}    
\end{equation}
is finite due to the assumptions regarding~$\tD$ for a small $\zeta>0$.
The inclusion of the $\zeta$-balls essentially circumvent problems related to the non-differentiability of~$\gamma_i$ in the proximity of~$(\vx_i,t_i)$.

We note that by construction $\gamma$ is a unit-speed geodesic for the length functional~\eqref{eq:lengthFunctional}.
In this case, we define the exponential map in the direction $\overline t\dot{\gamma}(0)$ as $\operatorname{Exp}_{\vx}(\overline t\dot{\gamma}(0))=\gamma(\overline{t})$, and the logarithm $\operatorname{Log}_{\vx}(\gamma(\overline{t}))=\overline t\dot{\gamma}(0)$ as its inverse.
Note the the logarithm can efficiently be computed by tracking backward the geodesic from $\gamma(\overline{t})$ to $\vx$.
\begin{remark}
We note that points can belong to multiple regions of influence, at which the derivative of~$\phi_\E$ might not be defined.
However, due to the general functional-analytic setting the Lebesgue measure of these points is negligible.
\label{rmk:regions_of_influence_cut}
\end{remark}

\subsection{\TheMethod Algorithm}
\label{sec:geodesic_opt}
In this section, we introduce the \TheMethod Algorithm to solve~\eqref{eq:initial_opt_problem} using a gradient-based approach.
Here, we restrict to the specific functional
\begin{equation}
\mathcal{J}(\phi)\coloneqq\int_\Gamma\frac{1}{2}(\phi(\vx)-\widehat{\phi}(\vx))^2\diff{\vx},
\end{equation}
where $\Gamma\subset\Omega$ is a subdomain of~$\Omega$ with a positive Lebesgue measure and $\widehat{\phi}\in L^2(\Gamma,\R)$ is a fixed square-integrable function.
In numerical experiments, $\widehat{\phi}$ reflects the measurements on a known subdomain $\Gamma$, for which
the quadratic mismatch on $\Gamma$ between $\phi$ and $\widehat{\phi}$ with respect to the EASs is minimized.
Examples of $\Gamma$ include finite sets of points mimicking a contact recording map, the full endocardium/epicardium, or subregions of them.

According to \Cref{sec:intro_min_problem,sec:geodesic_extraction}, the optimization problem for $N=1$ simply reads
\begin{equation}
\min_{(\vx_1,t_1)\in\mathcal{U}_1}\int_{\Gamma}\frac{1}{2}(t_1+\delta(\vx_1,\vx)-\widehat{\phi}(\vx))^2\diff{\vx}.
\label{eq:opt_problem}
\end{equation}
To employ a gradient-based approach, we see that following \Cref{prop:deriv_dist} the gradient of $\mathcal{J}$ with respect to $\vx_1$ simply reads as
\begin{equation}
\nabla_{\vx_1}\mathcal{J}(\phi_{\{(\vx_1,t_1)\}})=-\int_{\Gamma}r(\vx,\vx_1,t_1)\frac{\tD^{-1}(\vx_1)\dot{\gamma}_{\vx_1\to\vx}(0)}{\norm{\dot{\gamma}_{\vx_1\to\vx}(0)}_{\tD^{-1}(\vx_1)}}\diff{\vx},
\label{eq:source_gradient}
\end{equation}
where $\gamma_{\vx_1\to\vx}(t)$ is the geodesic path from $\vx_1$ to $\vx$ and $r(\vx,\vx_1,t_1)=t_1+\delta(\vx_1,\vx)-\widehat{\phi}(\vx)$ is the residual. 
Optimizing multiple points simultaneously yields an average direction weighted by the residuals $r$ on $\Gamma$.
\Cref{fig:opt_procedure} depicts how the velocity field shown in \Cref{fig:geodesics} translates to a descent direction to optimize \eqref{eq:opt_problem}.

\begin{figure}[htb]
\centering
    \includegraphics{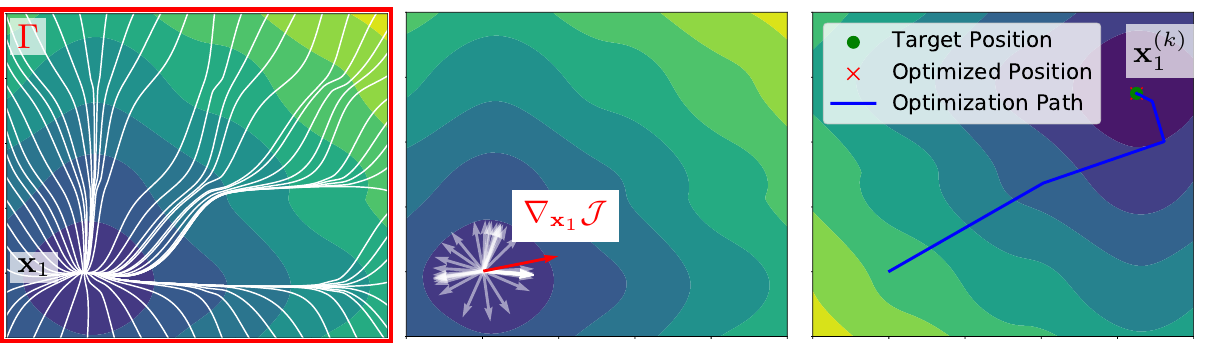}
\caption{Visualization of the optimization problem in \eqref{eq:opt_problem}. 
Geodesics (white) originating from the single EAS $\vx_1$ to distinct points on~$\Gamma$ (left)
and corresponding gradients computed with~\eqref{eq:source_gradient_wo_residuals} (middle).
The highlighted direction (red) coincides with the gradient in~\eqref{eq:source_gradient}.
Right: by iteratively applying a gradient-based scheme we determine the optimal~$(\vx_1,t_1)$.}
\label{fig:opt_procedure}
\end{figure}

The extension to multiple EASs works similarly.
A convenient formulation consists in splitting $\Gamma$ into subdomains $\Gamma_i\coloneqq\mathcal{R}_i\cap\Gamma$, each composed of those points activated by the EAS $\vx_i$ (note that the set of points belonging to multiple regions~$\Gamma_i$ has Lebesgue measure~$0$).
Then, the objective function reads as follows
\begin{equation}\label{eq:opt_problem_multi}
\min_{\{(\vx_i,t_i)_{i=1}^N\}\in\overline{\mathcal{U}_N}}\sum_{i=1}^N\int_{\Gamma_i}\frac{1}{2}(t_i+\delta(\vx_i,\vx)-\widehat{\phi}(\vx))^2\diff{\vx}.
\end{equation}
Clearly, the optimization procedure for a single EAS readily translates to the case of multiple sites.

We found that rather than a simple gradient descent scheme, a Gauss--Newton optimization proved beneficial to reduce the overall number of required optimization iterations, resulting in the following update rule
\begin{align}
\E^{(k+1)}=&\argmin_{\{(\vx_i,t_i)_{i=1}^N\}\in\overline{\mathcal{U}_N}}\notag\\
&\sum_{i=1}^N\frac{1}{2}\norm{\nabla_{\vx_i,t_i}\mathcal{J}(\phi_{\E^{(k)}})
\begin{pmatrix}
\vx_i-\vx_i^{(k)}\\
t_i-t_i^{(k)}
\end{pmatrix}
+\phi_{\E^{(k)}}(\vx)-\widehat{\phi}(\vx)}_{L^2(\Gamma)}^2.
\label{eq:updateEK}
\end{align}
Here, $\E^{(k)}=\{(\vx_i^{(k)},t_i^{(k)})_{i=1}^N\}$ are the solutions of the previous iteration.
To overcome local minima of the optimization problem~\eqref{eq:initial_opt_problem} caused by non-unique solutions (see \Cref{rmk:bounded})
we additionally use an over-relaxation~\cite{chambolle_introduction_2016,pock_inertial_2016} with fixed $\beta_a=\frac{1}{\sqrt{2}}$.
The resulting \Cref{alg:geasi_phi} iteratively linearizes and solves the problem using the computed gradient from $\delta$ to match a given measured activation.
We remark that the gradient properly reflects infinitesimal changes of activation times on $\mathcal{R}_i$ for each $\vx_i$, but it is not capable of accurately capturing
higher order effects like the change of $\mathcal{R}_i$. 
The experiments showed that rather than directly using $\E^{(k+1)}$ from \eqref{eq:updateEK} as the new solution, it is beneficial to take a step-size $\beta_s < 1$ and compute the convex combination of old and new solution according to this step size.
\label{rev:overrelaxation}
For all experiments, we used $\beta_s = \frac{1}{2}$.
For further details of the numerical realization we refer the reader to~\Cref{sec:discretization}.

\RestyleAlgo{boxruled}
\begin{algorithm}[htb]
 \SetKwInOut{Input}{Input}\SetKwInOut{Output}{Output}
 \Input{initial $\vx_i^{(0)}$ and $t_i^{(0)}$ defining $\E^{(0)}=\{(\vx_i^{(0)},t_i^{(0)})_{i=1}^N\}$,\\
        target activation $\widehat{\phi}(\vx)$ for $\vx\in\Gamma$, conduction velocity tensor $\tD$}
 \Output{optimal EASs $\vx_i^*$ and times $t_i^*$}
 \For{$k = 1, \ldots, K$}{
  $\widetilde{\E}^{(k)} = \E^{(k)} + \beta_a (\E^{(k)} - \E^{(k-1)})$ \\
  solve the eikonal equation \eqref{eq:eikonal} for~$\widetilde{\E}^{(k)}=\{(\widetilde{\vx}_i^{(k)},\widetilde{t}_i^{(k)})_{i=1}^N\}$\\
  compute all geodesics $\gamma_{\vx_i\to\vx}(t)$ for $\vx\in\Gamma$ by solving \eqref{eq:ode_eikonal_geodesics} \\
  compute $\bar{\E}^{(k+1)}$ using \eqref{eq:updateEK} (with $\widetilde{\E}^{(k)}$) \\
  $\E^{(k+1)} = \widetilde{\E}^{(k)} + \beta_s \left( \bar{\E}^{(k+1)} - \widetilde{\E}^{(k)} \right)$ \\
 }
 \caption{\TheMethod}
 
 \label{alg:geasi_phi}
\end{algorithm}

\section{Extensions of \TheMethod}
\label{sec:methods_extension}
\TheMethod is a versatile optimization algorithm, which can be extended in several aspects.
In this section, we focus on two such possible extensions. 
First, the topological gradient estimation allows for an accurate estimation of the number of EASs.
Second, we modify the original objective function of \TheMethod to fit a given ECG.

\subsection{Variable number of EASs: Topological Gradient}
\label{sec:topo_grad_method}
So far, we assumed the number of EASs $N$ to be fixed.
Since the optimal number of EASs is in general unknown, we subsequently propose a method to approximate the optimal~$N$.
As a possible approach to estimate~$N$ (which is not conducted in this work) one could start with a large number of sites and successively remove distinct EASs that violate the constraint~\eqref{eq:compatibilityCondition}.
However, this approach suffers from some major drawbacks:
\begin{itemize}
\item several local minima can occur leading to a strong dependency on the initial guess,
\item enforcing~\eqref{eq:compatibilityCondition} results in some numerical issues, e.g.~dimension changes of the optimization problem and order of EAS removal.
\end{itemize}
In contrast, starting with a few (or even a single) EASs and subsequently introducing new EASs overcomes the above issues
since according to \Cref{prop:decreasing} adding new sites does not increase the objective function.
In what follows, we briefly recall the topological gradient, which is used to compute the infinitesimal expansion of splitting a single EAS.
\label{rev:R1energy} This expansion is exploited to estimate the decrease in objective function of adding a new site.

Consider the case of a single EAS, i.e.~$N=1$.
The topological gradient is defined as the effect on the solution of the associated eikonal equation
if splitting a single EAS $\vx_1$ into two new sites $\vx_1+\varepsilon\mathbf{n}$ and $\vx_1-\varepsilon\mathbf{n}$ in the direction of $\mathbf{n}\in\sphere^{d-1}$.
We can directly infer from~\eqref{eq:mineiko} that
\begin{equation}\label{eq:topo_two_sites}
\phi_{\E_\varepsilon}(\vx)=\min\bigl\{t_1+\delta(\vx_1-\varepsilon\mathbf{n},\vx),t_1+\delta(\vx_1+\varepsilon\mathbf{n},\vx)\bigr\}
\end{equation}
for $\E_\varepsilon=\{(\vx_1+\varepsilon\mathbf{n},t_1),(\vx_1-\varepsilon\mathbf{n},t_1)\}$, where $\varepsilon>0$ is sufficiently small.
This topological operation divides the domain into two subdomains $\Omega_\varepsilon^-\coloneqq\{\vx\in\Omega:\delta(\vx_1-\varepsilon\mathbf{n},\vx)<\delta(\vx_1+\varepsilon\mathbf{n},\vx)\}$ and $\Omega_\varepsilon^+=\Omega\setminus\Omega_\varepsilon^-$.
We can now expand $\phi_\varepsilon$ with respect to $\varepsilon$ as follows
\begin{align*}
\phi_{\E_\varepsilon}(\vx)&=
t_1+\delta(\vx_1,\vx)+\varepsilon\min\{-\nabla_{\vx_1}\delta(\vx_1,\vx)\cdot\mathbf{n},
\nabla_{\vx_1}\delta(\vx_1,\vx)\cdot\mathbf{n} \} +o(\varepsilon)\\
&=\Phi_1(\vx,\vx_1,t_1)-\varepsilon\abs{\nabla_{\vx_1}\delta(\vx_1,\vx)\cdot\mathbf{n}}+o(\varepsilon),
\end{align*}
where $\Phi_1$ was defined in \eqref{eq:PhiNDefinition} and we used~\eqref{eq:source_gradient_wo_residuals}.
In this case, we call the quantity
\begin{equation}
j(\vx,\vx_1,\mathbf{n})\coloneqq-\abs{\nabla_{\vx_1}\delta(\vx_1,\vx)\cdot\mathbf{n}}
\label{eq:topo_deriv}
\end{equation}
the \emph{topological gradient}.

\begin{figure}[tb]
    \resizebox{\linewidth}{!}{
    \includegraphics{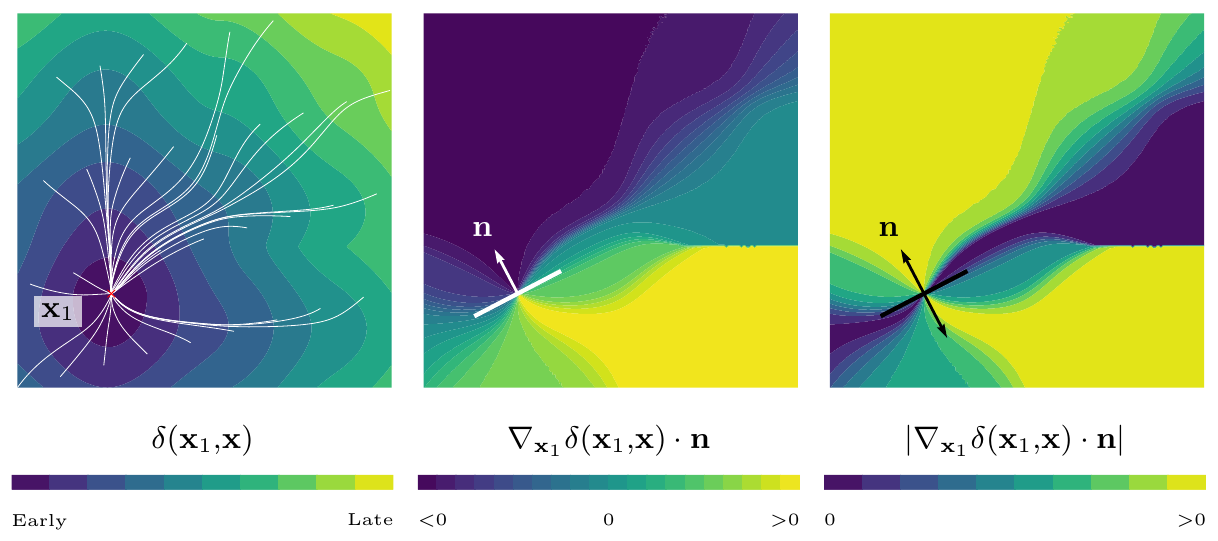}
    }
    \vspace{-.75cm}
    \caption{
    Left: geodesics (white) joining multiple points with~$\vx_1$.
    Contour plots of $\vx\mapsto\nabla_{\vx_1}\delta(\vx_1,\vx)\cdot\mathbf{n}$ (middle)
    and $\vx\mapsto \abs{\nabla_{\vx_1}\delta(\vx_1,\vx)\cdot\mathbf{n}}$ (right) for fixed $\mathbf{n}\in\sphere^{d-1}$.
    Moving a single EAS in the direction~$\mathbf{n}$ alters the activation times~$\delta(\vx_1,\vx)$ as shown.
    In contrast, splitting in the same direction~$\mathbf{n}$ is similar to simultaneously moving a source point in both directions, and keeping only the shorter geodesic (right).}
    \label{fig:asymptotic_changes}
\end{figure}

A visual example of the topological gradient is provided in \Cref{fig:asymptotic_changes}.
\label{rev:R2topo}
It is worth noting that the activation $\phi_{\E_\varepsilon}$ continuously depends on the splitting distance $\varepsilon$. 
Hence, the topological operation of splitting an EAS does not introduce any discontinuities in the objective function.
Moreover, we note that adding new optimal sites always decreases the objective functional unless $\nabla_{\vx_1}\delta(\vx_1,\vx)\cdot\mathbf{n} = 0$.
Therefore, we shall define a criterion for adding a split.
The decrease in objective function of splitting a single site can be estimated as follows:
\begin{equation}
\nu_{S,\varepsilon}\coloneqq\min_{\mathbf{n}\in\sphere^{d-1}}
\int_\Gamma r(\vx,\vx_1,t_1)^2
- \bigl(r(\vx,\vx_1,t_1) + \varepsilon j(\vx,\vx_1,\mathbf{n})\bigr)^2 \diff{\vx},
\label{eq:min_split}
\end{equation}
where $r(\vx,\vx_1,t_1) = \Phi_1(\vx,\vx_1,t_1)-\widehat{\phi}(\mathbf{x})$.
Likewise, the effect of moving a source point in direction~$\mathbf{n}$ is given by
\begin{equation}
\nu_{M,\varepsilon}\coloneqq\min_{\mathbf{n}\in\sphere^{d-1}}
\int_\Gamma r(\vx,\vx_1,t_1)^2
-\bigl( r(\vx,\vx_1,t_1) + \varepsilon\nabla_{\vx_1}\delta(\vx_1,\vx)\cdot\mathbf{n}\bigr)^2 \diff{\vx}.
\end{equation}
The ratio $\frac{\nu_{M,\varepsilon}}{\nu_{S,\varepsilon}}$ has proven to be a robust score for adding new sites, which is verified in the numerical experiments.
In particular, if the ratio is below a certain threshold, then a new EAS is introduced.

\subsection{Optimization using the ECG}
\label{sec:ecg_method}

The electrocardiogram (ECG) is the observed signature of the electric activity
of the heart, which is measured at selected locations on the chest.  Being
routinely acquired and non-invasive, the ECG is the ideal candidate for
inferring cardiac activation in a clinical framework.  Here, we will introduce
a method to reconstruct the EASs directly from ECG measurements.  To this end,
we exploit the methods presented
in~\cite{neic_efficient_2017,pezzuto_evaluation_2017} to efficiently compute
the ECG from activation maps of the eikonal equation.  Finally, the quadratic
mismatch of the computed and measured ECG is minimized, which yields optimal
EASs.

\begin{figure}[htb]
\centering
\includegraphics{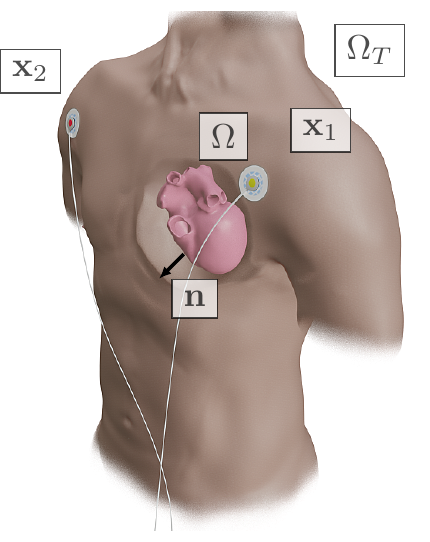}
\caption{Exemplary setup for the lead-field of Lead~I appearing in many ECG recordings.
$\Omega_T$ (torso) encapsulates heart domain $\Omega$ (heart domain), the left/right arm electrodes on $\Omega_T$ are marked as $\vx_{1}$ and $\vx_{2}$.
The normal $\mathbf{n}$ is pointing from the surface of the heart domain $\Omega$ into the torso domain $\Omega_T$.
The corresponding $Z_l$ is computed using~\eqref{eq:lead_field}, which is subsequently employed to obtain the ECG~$V_l$ of Lead~I.}
\label{fig:torso_lead_field_demo}
\end{figure}

From a modeling perspective, we denote by $\Omega_T\subset\R^d \setminus \overline\Omega$ the whole body domain excluding the heart cavity.
The heart-torso interface $\Gamma_H \coloneqq\overline{\Omega_T}\cap\overline\Omega$
is the boundary between the active myocardium and the rest of the body (for instance, endocardium plus epicardium),
whereas $\Sigma\coloneqq\partial\Omega_T\setminus\Gamma_H$ is the chest, on which the aforementioned electrical signal is recorded.
We denote by $\mathbb{T}\subset\R$ the considered time interval.
In \Cref{fig:torso_lead_field_demo}, we outline how this setup would look in actual clinical measurements on the example of Lead~I.

\label{rev:R1bidomain}
An equation for the torso potential can be derived from bidomain theory for the cardiac tissue and the balance of currents in the body (see e.g.~\cite{colli2014}).
Here, we consider the
so-called \emph{pseudo-bidomain}~\cite{bishop2011} or \emph{forward-bidomain} model~\cite{potse2006}, in which the parabolic and elliptic part of the bidomain equation
are decoupled and can be solved sequentially.  In this way, the transmembrane potential, denoted by $\Vm(\vx,t,\E)$, is not affected by the extracellular and torso potentials and can therefore be approximated independently.  The resulting system of equations reads as follows:
\begingroup
\arraycolsep=1.4pt\def\arraystretch{1.2}
\begin{equation}\label{eq:bidomain_torso}
\left\{
\begin{array}{rcll}
-\nabla\cdot(\mathbf{G}\nabla u_e)&=&\nabla\cdot(\mathbf{G}_i\nabla\Vm), &\text{in }\Omega\times\mathbb{T},\\
-\nabla\cdot(\mathbf{G}_T\nabla u_T)&=&0, &\text{in }\Omega_T\times\mathbb{T},\\
-\mathbf{G}_T\nabla u_T\cdot\mathbf{n}&=&0,&\text{in }\Sigma\times\mathbb{T},\\
u_e(\vx^-,t)&=&u_T(\vx^+,t),&(\vx,t)\in\Gamma_H\times\mathbb{T},\\
\mathbf{G}_T(\vx^+)\nabla u_T(\vx^+,t)\cdot\mathbf{n}\qquad&\\
-\mathbf{G}(\vx^-)\nabla u_e(\vx^-,t)\cdot\mathbf{n}&=&\mathbf{G}_i(\vx^-)\nabla\Vm(\vx^-,t)\cdot\mathbf{n},&(\vx,t)\in\Gamma_H\times\mathbb{T},
\end{array}
\right.
\end{equation}
\endgroup
where the following quantities occur:
\begin{itemize}[label=-]
    \setlength\itemsep{0em}
\item
$u_e(\cdot,\cdot,\E)\colon\Omega\times\mathbb{T}\to\R$ is the extracellular potential in the heart, parametrized through the set of EASs $\E$,
\item
$\Vm(\cdot,\cdot,\E)\colon\Omega\times\mathbb{T}\to\R$ is the transmembrane potential,
\item
$u_T(\cdot,\cdot,\E)\colon\Omega_T\times\mathbb{T}\to\R$ is the potential in the torso,
\item
$\mathbf{G}_T\colon\Omega_T\to\Sym{d}$ is the electric conductivity of the torso,
\item
$\mathbf{G}_i\colon\Omega\to\Sym{d}$ is the intracellular conductivity,
\item
$\mathbf{G}_e:\Omega\to\Sym{d}$ is the extracellular conductivity, and
\item
$\mathbf{G}=\mathbf{G}_i+\mathbf{G}_e$ is the bulk conductivity of the heart.
\end{itemize}
The normal vector $\mathbf{n}$ at $\vx\in\Gamma_H$ points outwards, i.e.~from the heart surface towards the torso,
and is the outer normal vector for $\vx\in\Sigma$.
The points $\vx^{\pm}$ associated with $\vx\in\Gamma_H$ are obtained by taking the limit $\vx^\pm_\varepsilon=\vx\pm\varepsilon\mathbf{n}$ for $\varepsilon\to0$.

The well-posedness of~\eqref{eq:bidomain_torso} follows from standard arguments for elliptic PDEs (see~\cite{Gilbarg01}).
However, some care is required for the discontinuity across the heart-torso interface $\Gamma_H$.
Indeed, in order to render~\eqref{eq:ecg_bidomain} meaningful, we need at least $u_T$ to be continuous on~$\Sigma$ for every $t\in \mathbb{T}$.
Let $\widetilde{\Omega}=\Omega\cup\Omega_T\cup\Gamma_H$ be the domain modeling the whole torso (including the heart), and
\[
\widetilde{\mathbf{G}}=
\begin{cases}
\mathbf{G},&\text{in }\Omega,\\
\mathbf{G}_T,&\text{in }\Omega_T,
\end{cases}
\qquad
\widetilde{u}=
\begin{cases}
u_e,&\text{in }\Omega\times\mathbb{T},\\
u_T,&\text{in }\Omega_T\times\mathbb{T}.
\end{cases}
\]
Following~\cite{colli2014}, we assume that
\begin{enumerate}
\item $\Omega,\Omega_T\subset\R^d$ are Lipschitz domains,
\item $\mathbf{G}_i,\mathbf{G}_e\in C^1(\Omega,\Sym{d})$ and $\widetilde{\mathbf{G}}\in L^\infty(\widetilde{\Omega},\Sym{d})$,
\item $\Vm(\cdot,t,\E)\in W^{2,p}(\Omega)$, with $p>d$,
for all $t\in\mathbb{T}$ and $\E$.
\end{enumerate}
\begin{proposition}
Under the above assumptions, the weak formulation of \eqref{eq:bidomain_torso} given by
\begin{equation}\label{eq:bidomain_torso_var}
\begin{split}
&\text{find }\widetilde{u}(\cdot,t,\E)\in H^1(\widetilde{\Omega})\text{ s.t. } \\
&\int_{\widetilde{\Omega}}\widetilde{\mathbf{G}}\nabla\widetilde{u}(\vx,t,\E)\cdot\nabla v\diff{\vx}
= -\int_\Omega\mathbf{G}_i\nabla\Vm(\vx,t,\E)\cdot \nabla v\diff{\vx},
\quad\forall v\in H^1(\widetilde{\Omega})
\end{split}
\end{equation}
is well-defined.
In particular, there exists a unique solution up to an additive constant (the reference potential).
\end{proposition}
The proof directly follows from the Lax--Milgram theorem~\cite{Gilbarg01} by noting that
$\Vm(\cdot,t,\E)\in W^{2,p}(\Omega)$ and $\widetilde{u}(\cdot,t,\E)\in W^{1,p}(\widetilde{\Omega})$ for $p>d$ and all $t\in\mathbb{T}$.

Let $\vx_e\in\Sigma$, $e=1,\ldots,N_E$, be $N_E$ electrodes placed on the chest.
A single-lead ECG recording is the potential difference between two such electrodes or, more generally, a zero-sum linear combination of the recordings.
For instance, Einthoven's lead~I is the potential difference between left and right arm electrodes.
More generally, given electric potentials $u_T$ at the electrodes, the standard ECG is a vector-valued function $\mathbf{V}\colon\mathbb{T}\to\R^L$ given by 
\begin{equation}\label{eq:ecg_bidomain}
\mathbf{V}(t,\E)=
\begin{pmatrix}
V_1(t,\E) \\ V_2(t,\E) \\ \vdots \\ V_L(t,\E)
\end{pmatrix}=
\mathbf{A}
\begin{pmatrix}
u_T(\vx_1,t,\E) \\ u_T(\vx_2,t,\E) \\ \vdots \\ u_T(\vx_{N_E},t,\E)
\end{pmatrix},
\end{equation}
where $L$ is the number of leads and $\mathbf{A}$ is a $L\times N_E$ real matrix defining the lead system, e.g.~the 12-lead ECG.
Since each row of $\mathbf{A}$ sums to zero, the matrix is not full-rank.
For instance, the standard 12-lead ECG corresponds to the choice $L=12$ (3 Einthoven leads, 3 augmented limb leads and 6 precordial leads) and $N_E=9$ (3 limb electrodes and 6 precordial electrodes), in which case $\mathbf{A}$ has rank 8.
We remark that Morrey's inequality guarantees $\widetilde{u}(\cdot,t,\E)\in C^0(\widetilde{\Omega})$, hence validating~\eqref{eq:ecg_bidomain}.

Solving~\eqref{eq:bidomain_torso} is numerically costly for the standard 12-lead ECG, since we only evaluate~$u_T$ at selected locations.
Note that the system must be solved for every~$t\in\mathbb{T}$.
Thus, we adopt the following integral representation of~\eqref{eq:ecg_bidomain}:
\begin{equation}\label{eq:ecg_bidomain_int}
V_l(t,\E)=\int_\Omega\mathbf{G}_i(\vx)\nabla\Vm(\vx,t,\E)\cdot\nabla Z_l(\vx)\diff{\vx},
\end{equation}
where $Z_l\colon\widetilde{\Omega}\to\R$ are the lead fields (or Green's functions) satisfying the adjoint problem
\begin{equation}\label{eq:lead_field}
\left\{
\begin{array}{rcll}
-\nabla\cdot(\widetilde{\mathbf{G}}\nabla Z_l)&=&0, &\text{in $\Omega\cup\Omega_T$}, \\
-\widetilde{\mathbf{G}}\nabla Z_l(\vx)\cdot\mathbf{n} &=& \displaystyle \sum_{e=1}^{N_E}[\mathbf{A}]_{le}\delta_{[\vx-\vx_e]},&\vx\in\Sigma,\\
Z_l(\vx^-)&=&Z_l(\vx^+),&\vx\in\Gamma_H,\\
\mathbf{G}(\vx^+)\nabla Z_l(\vx^+)\cdot\mathbf{n}&=&\mathbf{G}(\vx^-)\nabla Z_l(\vx^-)\cdot\mathbf{n},&\vx\in\Gamma_H.
\end{array}
\right.
\end{equation}
An informal derivation of~\eqref{eq:ecg_bidomain_int} follows from the application of the second Green's identity to~\eqref{eq:bidomain_torso}.  As for~\eqref{eq:bidomain_torso},
the solution is defined up to a constant.
For a rigorous derivation accounting for the discontinuity in~$\widetilde{\mathbf{G}}$, we refer the reader to~\cite[pp.~152~ff.]{colli2014}.  A key observation is that the lead
fields do not depend on $t$ and $\E$, making~\eqref{eq:ecg_bidomain_int} particularly
attractive for parameter estimation.

Next,
we assign the transmembrane potential $\Vm$ accordingly to a fixed waveform $U\colon\R\to\R$ shifted by the activation time~$\phi_\E$ as follows
\[
\Vm(\vx,t,\E)=U\bigl(t-\phi_\E(\vx)\bigr).
\]
We write the parametrized waveform as:
\begin{equation}
U(\xi)=K_0+\frac{K_1-K_0}{2}\left[\tanh\left(2\frac{\xi}{\tau_1}\right)-\tanh\left(2\frac{\xi-\text{APD}}{\tau_2}\right)\right],
\label{eq:tanh_waveform}
\end{equation}
which is visualized in \Cref{fig:tanh_waveform}.

\begin{figure}[htb]
    \centering
    \includegraphics{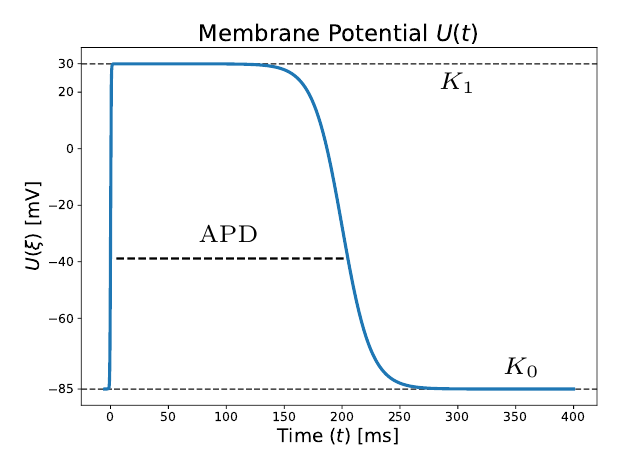}
    \caption{Membrane voltage waveform as a function of time, equivalent to \eqref{eq:tanh_waveform} with parameters from~\Cref{tab:tanh_waveform_parameters}. 
    The continuous formulation allows for an analytical derivation in~\eqref{eq:ecg_deriv}.}
    \label{fig:tanh_waveform}
\end{figure}
\begin{table}[htb]
    \centering
    \begin{tabular}{llll}
    \toprule
    Parameter & Description & Value & Unit\\
    \midrule
    $t$       & time & $\left[0, T\right]$ & \si{\ms} \\
    $\lambda$ & anisotropy ratio & 3 & --- \\
    $\mathbf{G}_T$ & torso conductivity & $0.2$ & \si{\milli\siemens\per\mm} \\
    $\alpha^2$ & conduction velocity scaling & 400 & \si{\square\mm\per\milli\siemens\per\square\ms} \\
    $\beta$   & surface-to-volume ratio & 100 & \si{\per\mm} \\
    $K_0$     & resting potential & $-85$ & \si{\milli\volt} \\
    $K_1$     & plateau potential & 30 & \si{\milli\volt} \\
    $\tau_1$  & depolarization time-scale & 1 & \si{ms} \\
    $\tau_2$  & repolarization time-scale & 50 & \si{ms} \\
    APD       & action potential duration & 200 & \si{ms} \\
    \bottomrule
    \end{tabular}
    \caption{Parameters to compute the ECG from the eikonal solution $\phi_{\E}$.}
    \label{tab:tanh_waveform_parameters}
\end{table}

\label{rev:R1aniso} 
Furthermore, the conduction velocity tensor $\tD$ in the anisotropic eikonal equation in~\eqref{eq:eikonal} is linked to the electric conductivity as follows:
\begin{equation}
    \tD = \frac{\alpha^2}{\beta}\mathbf{G}_e \mathbf{G}^{-1} \mathbf{G}_i,
    \label{eq:diff_tensor_cond_link}
\end{equation}
where $\beta$ is the surface-to-volume ratio and $\alpha$ is a rescaling factor either experimentally estimated or obtained by solving the monodomain equation in a cable propagation setup~\cite{pezzuto_evaluation_2017}.
Note that in all conducted experiments we assumed an equal anisotropy ratio $\mathbf{G}_i = \lambda \mathbf{G}_e$, from which $\tD = \frac{\alpha^2}{\beta} \frac{\lambda}{1 + \lambda}\mathbf{G}_i$ follows.
The equal anisotropy ratio assumption simplifies the numerical experiments, but is not necessary for \TheMethod.
All parameters adopted in this study are provided in \Cref{tab:tanh_waveform_parameters}. 

We emphasize that $\phi_\E\in C^{0,1}(\overline{\Omega})$ only implies $\Vm(\cdot,t,\E)\in W^{1,p}(\Omega)$ and not $\Vm(\cdot,t,\E)\in W^{2,p}(\Omega)$ as required for $t\in\mathbb{T}$ and $\E\in\mathcal{U}_N$.
However, the aforementioned theory is still valid in this case with some major modifications that are beyond the scope of this work.
Again, we refer to~\cite{colli2014} and the references therein for further details.

In what follows, we intend to compute the sensitivities of the ECG with respect to the parameter set $\E\in\mathcal{U}_N$.
In the problem, only the activation map~$\phi_\E$ appearing in the definition of $\Vm$ depends on the parameters in $\E$.
Note that the chain rule straightforwardly implies
\begin{equation*}
\nabla_\E \Vm=-\frac{\partial U}{\partial \xi}\nabla_\E\phi_\E.
\end{equation*}
The use of the aforementioned smooth waveform allows for a continuous analytical derivative $\frac{\partial U}{\partial\xi}$. 
Details on the derivation of the term $\nabla_\E\phi_\E$ were already given in \Cref{sec:geodesic_extraction}.
Then, the derivative $\nabla_\E V_l$ is computed from \eqref{eq:ecg_bidomain_int} and reads as
\begin{equation}
\nabla_\E V_l(t,\E)=\int_\Omega\left(\mathbf{G}_i(\vx)\nabla_{\E,\vx}^2 \Vm(\vx,t,\E)\right)\nabla Z_l(\vx)\diff{\vx}.
\label{eq:ecg_deriv}
\end{equation}
Finally, in this model the set of EASs~$\E$ is computed from the measured ECG $\widehat{V}_l:I \to\R$ as follows:
\begin{equation}
\min_{\E\in\overline{\mathcal{U}_N}}\frac{1}{2}\sum_{l=1}^L\int_{\mathbb{T}}\left(V_l(t,\E)-\widehat{V}_l(t)\right)^2\diff{t},
\label{eq:ecg_min}
\end{equation}
which is solved using the Gauss--Newton algorithm in a similar fashion to~\Cref{alg:geasi_phi}.
In particular, the update of the set~$\E$ reads as follows
\begin{align}
&\E^{(k+1)}=\argmin_{\{(\vx_i,t_i)_{i=1}^N\}\in\overline{\mathcal{U}_N}}\notag\\
&\sum_{l=1}^L\sum_{i=1}^N\frac{1}{2}\norm{\nabla_{\vx_i,t_i}\mathcal{J}(\E^{(k)})(\vx_i-\vx_i^{(k)},t_i-t_i^{(k)})^\top+V_l(t,\E^{(k)})-\widehat{V}(t)}_{L^2(\mathbb{T})}^2
\label{eq:EUpdate}
\end{align}
with the modified objective functional
\begin{equation}
\mathcal{J}(\E)=\frac{1}{2} \sum_{l=1}^L \int_{\mathbb{T}} \bigl( V_l(t,\E)-\widehat{V}_l(t)\bigr)^2\diff{t}.
\end{equation}
The numerical integration in~\eqref{eq:EUpdate} is realized using the trapezoidal rule.

\begin{remark}
There are several numerical issues related to the optimization:
\begin{enumerate}
\item 
The waveform~\eqref{eq:tanh_waveform} is a rough approximation of a physiological action potential modelled the electrophysiology of a cell.
The function~$U$ and the scaling parameter~$\alpha$ may be simultaneously approximated from a generic ionic model by solving a 1-D propagation in a
(possibly very long) cable with uniform coefficients.
Alternatively, it is possible to show that $(U,\alpha)$ solves a nonlinear eigenvalue problem involving the ionic model~\cite{keener1991eikonal}.
\item
\label{rev:R1repol}
\Cref{eq:tanh_waveform} is actually not suitable to model the repolarization of the heart which is responsible for the T-wave.
The reason is that the polarity of the T-wave, in general in accordance with the polarity of the QRS complex, can only arise from a heterogeneity in the action potential.
Such heterogeneity might be introduced here, but it would be hard to reproduce the smoothing effect due to diffusion currents.
Finally, the eikonal model is not suitable for the repolarization because,
opposed to the depolarization phase, the repolarization front is of the same order of the size of the domain, impeding a proper perturbation analysis.
In this work, the repolarization time is $\phi_\E(\vx)+\mathrm{APD}$, hence it satisfies the same equation as $\phi_\E$, but with a shifted time.
\item
\Cref{eq:ecg_deriv} requires higher order derivatives of $\Vm$ and subsequently $\phi_\E$. 
While we computed the derivative $\nabla_\E\Vm$ as previously discussed, the computation of $\nabla_\vx\Vm$ is numerically achieved on the reference element.
\item
It is important to mention that the gradient computation for the minimization of~\eqref{eq:ecg_min} is usually much more costly compared to optimizing the problem in the eikonal formulation from~\eqref{eq:opt_problem},
since the size of~$\Gamma$ is much smaller compared to~$\Omega$.
However, to compute $\nabla_{\vx_i,t_i}\mathcal{J}$ we need the activation times and their derivatives in~$\Omega$,
which necessitates the computation of the geodesics from each point of our domain to the EAS $\vx_i$.
The computational complexity is significantly larger than the complexity for~\eqref{eq:opt_problem}.
Further strategies to reduce additional computational costs are presented in \Cref{sec:runtime}.
\end{enumerate}
\end{remark}

\section{Discretization}\label{sec:discretization}
In this section, we elaborate on the discretization aspects for \Cref{alg:geasi_phi}, which encompasses the steps:
over-relaxation of $\E^{(k+1)}$, solving the eikonal equation, computation of the geodesics and update of $\E^{(k+1)}$.

\subsection{Solving the Eikonal Equation}\label{sec:eikonal}
The discrete function space for the eikonal equation is the space of volumetric Lagrange $\mathcal{P}^1$-finite elements defined on triangular ($d=2$) and tetrahedral ($d=3$) meshes discretizing~$\Omega$, respectively.
Moreover, the discrete measurements in~$\Gamma$ are degrees of freedom (DOFs) of the mesh.

Typically, finite element solvers require the initiation sites to coincide with DOFs of the mesh.
However, since the original problem~\eqref{eq:initial_opt_problem} expresses $\vx_i$ as a continuous quantity, we identify the DOFs of the actual element containing the activation site.
Then, these DOFs are added to the Dirichlet boundary~$\Gamma_D$ with fixed activation times given by $t_i+\norm{\vx-\vx_i}_{\tD^{-1}(\vx_i)}$ for $\vx\in\Gamma_D$
due to structural assumptions regarding the $\mathcal{P}^1$-finite element space.
For the rare case of two or more initiation sites residing in the same element, we use the properties of \eqref{eq:mineiko} and \eqref{eq:compatibilityCondition} to compute the activation times.

In all subsequent computations, we employ the FIM \cite{FIM_triangle,FIM_tetra} to solve the eikonal equation to account for the anisotropy.

\subsection{Computation of Geodesics}\label{sec:computationGeodesics}
In this work, we employ Heun's method (second order explicit Runge--Kutta scheme) to solve~\eqref{eq:ode_eikonal_geodesics}, which proved to be stable and efficient in numerical experiments.
Due to the convergence of the ODE system to a stable node~$\vx_i$ we terminate the iteration if the $\ell^2$-norm of two consecutive iterations is below $10^{-10}$.
In practice, the ODE system is solved independently on each region of interest~$\mathcal{R}_i$ incorporating the whole set~$\E$.
Since the gradient of the eikonal solution for the chosen discretization is a $\mathcal{P}^0$-finite element function (i.e.~piecewise constant), we advocate a standard $L^2$-projection onto the $\mathcal{P}^1$-finite element space~\cite{larson_finite_2013}.
Note that this projection can be realized by solving a linear system involving the mass matrix in~$\mathcal{P}^1$.
Since the boundary of~$\Omega$ is in general curved, we project the geodesics back onto $\partial\Omega$ if they are outside of the domain after each update.

As remarked in \Cref{sec:geodesic_extraction}, the gradient of the eikonal solution is discontinuous around each~$\vx_i$.
To enforce regular gradients at each $\vx_i$ after the $L^2$-projection of the previous eikonal solution~$\widetilde{\phi}_\E$, we recompute the points with vanishing gradient by the subsequent variational problem with Tikhonov regularization for $c\in\llbracket d\rrbracket$ and balancing parameter $\lambda>0$ as follows:
\begin{equation*}
\begin{pmatrix}
[\nabla\phi_\E(\vy_1)]_c\\
\vdots\\
[\nabla\phi_\E(\vy_{d+1})]_c\\
\end{pmatrix}
=\argmin_{\mathbf{n}\in\R^{d+1}}\frac{1}{2}\langle\mathbf{\Psi}(\vx_i),\mathbf{n}\rangle^2+\frac{\lambda}{2}\norm{\mathbf{n}-
\begin{pmatrix}
[\nabla\widetilde{\phi}_\E(\vy_1)]_c\\
\vdots\\
[\nabla\widetilde{\phi}_\E(\vy_{d+1})]_c\\
\end{pmatrix}
}^2.
\end{equation*}
Here, $\mathbf{\Psi}=(\psi_1,\ldots,\psi_{d+1})^\top$ and $\{\vy_j\}_{j=1}^{d+1}$ are the collections of
$\mathcal{P}^1$-basis functions and degrees of freedom associated with the element containing $\vx_i$, respectively.
\Cref{fig:x0_geodesic_accuracy} depicts the effect of this regularization on the solution around $\vx_i$.

\begin{figure}[htb]
\centering
\includegraphics{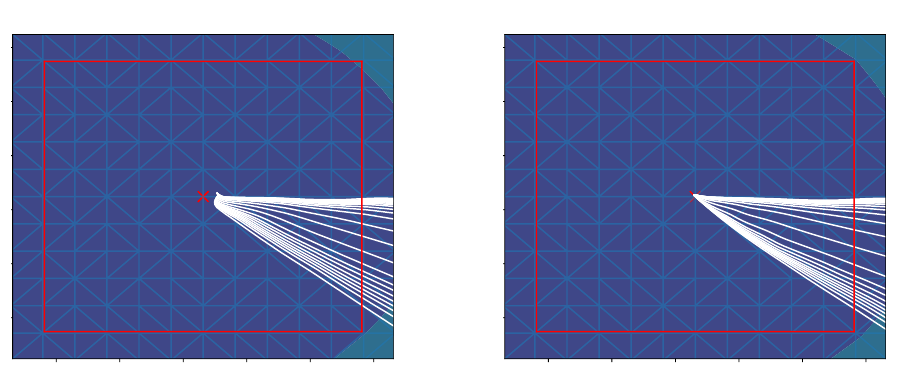}
\caption{Left: zoom of family of geodesics emanating from a single EAS without special handling.
Note that before the $L^2$-projection of $\nabla\phi_{\E}$ geodesics are not guaranteed to reach the EAS.
Right: after the $L^2$-projection all geodesic curves actually reach the EAS.}
\label{fig:x0_geodesic_accuracy}
\end{figure}
The gradient computation in~\eqref{eq:source_gradient_wo_residuals} is also sensitive to the choice of the step sizes of~\eqref{eq:ode_eikonal_geodesics}, which we choose as $5\cdot 10^{-2} h$ with $h$ being the average element size. 
As already described in \Cref{sec:geodesic_extraction}, we compute the geodesic direction not directly at $\vx_i$, but rather in a small $\zeta$-neighborhood with $\zeta=0.5h$ as advocated in~\eqref{eq:ZetaT}.
Numerically, the convergence of geodesics to this neighborhood is not ensured and non-converged geodesics (rarely occurring) do not affect the optimization.

\subsection{Update of $\E^{(k+1)}$}
\label{sec:update_ek}
Next, we optimize~\eqref{eq:updateEK}, where we have to ensure $\E\in\overline{\mathcal{U}_N}$.
The constraint $\vx_i\in\overline{\Omega}$ is mesh-dependant and allows for no general analytical solution, potentially limiting the available optimization implementations.
To overcome this hurdle, we use a proximal point algorithm enforcing $\E\in\overline{\mathcal{U}_N}$.
The integration is realized using an exact simplex quadrature rule.

In detail, we first compute the Moreau envelope of~\eqref{eq:updateEK} with respect to the metric induced by
$\mathbf{M}_{i,\E^{(k)}}\coloneqq\frac{1}{\tau}\Id-\mathbf{J}_{i,\E^{(k)}}^\top\mathbf{J}_{i,\E^{(k)}}$ for $\tau < \Vert\mathbf{J}_{i,\E^{(k)}}^\top\mathbf{J}_{i,\E^{(k)}}\Vert^{-1}$,
where $\mathbf{J}_{i,\E^{(k)}}\coloneqq\nabla_{(\vx_i,t_i)}\mathcal{J}(\vx_i^{(k)},t_i^{(k)})$.
Thus, the Moreau envelope reads as
\begin{equation}
f(\overline{\E})\coloneqq\min_{\E\in\overline{\mathcal{U}_N}}
\sum_{i=1}^N\frac{1}{2}\!
\norm{\mathbf{J}_{i,\E^{(k)}}\!
\begin{pmatrix}
\vx_i-\vx_i^{(k)}\\
t_i-t_i^{(k)}
\end{pmatrix}\!
+\mathbf{r}_{\E^{(k)}}(\vx)}_{L^2(\Gamma)}^2\!+\!
\frac{1}{2}\!\norm{
\begin{pmatrix}
\overline{\vx}_i-\vx_i\\
\overline{t}_i-t_i
\end{pmatrix}    
}_{\mathbf{M}_{i,\E^{(k)}}}^2\!
\label{eq:moreau_envelope}
\end{equation}
with $\E=\{(\vx_i,t_i)_{i=1}^N\}$ and $\overline{\E}=\{(\overline{\vx}_i,\overline{t}_i)_{i=1}^N\}$, and $\mathbf{r}_{\E^{(k)}}(\vx)=\phi_{\E^{(k)}}(\vx)-\widehat{\phi}(\vx)$.
This particular choice of the metric~\cite{chambolle_introduction_2016} allows for an explicit solution to~\eqref{eq:moreau_envelope} given the projection onto~$\mathcal{U}_N$.
This method is usually referred to as Iterative Shrinkage and Thresholding (ISTA, \cite{donoho_denoising_1995}).
In summary, the iteration step of the proximal point algorithm reads as
\begin{equation*}
\widehat{\E}= \left\{\left(\proj_{\mathcal{U}_1}\left(
\begin{pmatrix}
\overline{\vx}_i\\
\overline{t}_i
\end{pmatrix}    
-\int_\Gamma\tau\mathbf{J}_{i,\E^{(k)}}^{\top}\left(\mathbf{J}_{i,\E^{(k)}}
\begin{pmatrix}
\overline{\vx}_i-\vx_i^{(k)}\\
\overline{t}_i-t_i^{(k)}
\end{pmatrix} 
+\mathbf{r}_{\E^{(k)}}(\vx)\right)\diff{\vx}\right)\right)_{i=1}^N\right\}.
\end{equation*}
Thus, the resulting optimal set is $\widehat{\E}=\{(\widehat{\vx}_i,\widehat{t}_i)_{i=1}^N\}$.
Note that the convexity of the projection depends on the convexity of the domain.
In practice, hardly any cardiac mesh is convex, but nevertheless the proposed method generated reliable results for sufficiently small step sizes.
The gradient direction of the Moreau-envelope is
\begin{equation*}
\nabla f_{\vx_i,t_i}(\overline{\E})=\tau^{-1}
\begin{pmatrix}
\overline{\vx}_i-\widehat{\vx}_i\\
\overline{t}_i-\widehat{t}_i
\end{pmatrix} 
\end{equation*}
In this case, the unconstrained problem is solved using L-BFGS~\cite{BFGS}.

\section{Numerical Results}\label{sec:experiments}
Next, we present numerical results for various methods discussed above, where we focus on four setups to test \TheMethod on theoretical and cardiac problems:
\begin{enumerate}
\item[\fbox{$\mathbf{1}$}] The square domain presented in \Cref{fig:geodesics} with a periodic conduction velocity field, where the measurement domain $\Gamma$ coincides with the boundary of the surface.
                            The initial initiation sites were chosen randomly for optimization w.r.t.~activation times.
                            For the optimization w.r.t.~the ECG, we moved the target EASs further apart and used perturbations of the target EAS positions as initializations.
\item[\fbox{$\mathbf{2}$}] On a simplified 2D left ventricle (LV)-slice geometry with a transmural fiber rotation. 
Fiber and transverse intracellular conductivities were set to achieve
a conduction velocity of
\SIlist[list-units=single]{0.6;0.4}{\mm/\ms}, respectively.
The measurement domain~$\Gamma$ is the outer ring of the domain, i.e.~an epicardial slice.
The initial initiation sites were chosen randomly for all related experiments.
\item[\fbox{$\mathbf{3}$}] \label{rev:R5clinic} On a clinically sampled, endocardial electrical mapping, recorded during intrinsic rhythm in a patient candidate to cardiac resynchronization therapy (CRT) and affected by a left bundle branch block. 
Data acquisition and the construction of the patient-specific anatomical model has been described in previous studies~\cite{pezzuto_evaluation_2017,Pezzuto2021ECG}.
The measurements of activation were projected onto a patient-specific LV heart geometry.
We mapped fiber orientations into the model using the approach described in~\cite{Bayer12:_ld}. 
Fiber, transverse and cross conduction velocities were set
to \SIlist[list-units=single]{0.6;0.4;0.2}{\mm/\ms},
respectively.
Again, the initial initiation sites were chosen randomly, the target initiation sites are unknown.
\item[\fbox{$\mathbf{4}$}] A full biventricular, trifascicular LV/RV human heart geometry with $1000$ measurement points $\Gamma$ distributed evenly along the epicardium. 
Details on the model building process have been reported previously in~\cite{augustin_anatomically_2016}.
Conduction velocities and fibers were assigned as in the above experiment, with
an additional fast-conducting isotropic endocardial layer with a propagation
velocity of \SI{1.5}{\mm/\ms}.
As in almost all other experiments, the initial initiation sites were chosen randomly.
The target EASs, three in the left ventricle and three in the right ventricle,
were chosen in accordance with our previous study, see~\cite[p.~10]{grandits_inverse_2020}.
%
\end{enumerate}
\label{rev:R1cv} In all experiments involving the computation of the ECG and lead fields,
we assumed that intra- and extra-cellular conductivities are proportional to
the tensor $\tD$, as explained in \Cref{sec:ecg_method}.
Further numerical specifications of the aforementioned setups are listed in \Cref{tab:setups}
and \Cref{fig:exp_setups}.
\begin{table}[htb]
    \small
\begin{tabular}{|c|c|c|c|c|c|c|}
\toprule
 & \multirow{2}{*}{DOFs} & \multirow{2}{*}{Size [cm$^d$]} & $h$ & \multirow{2}{*}{Runtime [h]} & Topological & \multirow{2}{*}{ECG}\\
 & & & [mm] & & gradient & \\
\midrule
\fbox{$\mathbf{1}$} & $50^2$ & $2 \cdot 2$ & $0.4$ & 1/2 (ECG) & \checkmark & \checkmark \\
\fbox{$\mathbf{2}$} & $7980$ & $2 \cdot 2$ & $0.11$ & 1/2 (ECG) & \checkmark & \checkmark \\
\fbox{$\mathbf{3}$} & $1.5 \cdot 10^{4}$ & $10.7 \cdot 8.9 \cdot 9.5$  & $1.9$  & 2.5 & \checkmark &  \\
\fbox{$\mathbf{4}$} & $1.08 \cdot 10^{5}$ & $10.3 \cdot 8.1 \cdot 12.6$ & $0.66$  & 3.5/18 (ECG) & & \checkmark\\[.35em]
\hline
\end{tabular}
\caption{Selected parameters for each setup.
The size refers to the bounding box of the setups and $h$ is the average element spacing. 
Tested extensions are indicated by check marks.
The ECG runtimes are separately denoted behind the dash as the experiments are more computationally demanding.}
\label{tab:setups}
\end{table}

In this work, we use a custom \verb!C++! implementation of the Fast Iterative
Method~\cite{FIM_tetra} to solve the eikonal equation \eqref{eq:eikonal}.
Note however that the method is independent of the chosen eikonal solver and may benefit from higher order or smoother solutions of different solvers.
A minimal working example for the method can be found on GitHub\footnote{\url{https://github.com/thomgrand/geasi_grid_demo}}, but is limited to the isotropic eikonal equation on structured grids using the Fast Marching Method \cite{sethian_fast_1996,tsitsiklis_1995_efficient} without FEM.
The ECG and geodesic computations, i.e.~\eqref{eq:ecg_bidomain_int} and~\eqref{eq:ode_eikonal_geodesics}, and its Jacobian computation are calculated using the \textsc{TensorFlow} framework\footnote{\url{https://www.tensorflow.org/}}, making use of available GPUs and enabling automatic derivation of~$V_l$ with respect to~$\widetilde{\Vm}$. 
All computations were performed on a single desktop machine with an Intel Core i7-5820K CPU using 6 cores of each 3.30GHz, 32GB of working memory and a NVidia RTX 2080 GPU.

 \begin{figure}[htb]
\resizebox{\linewidth}{!}{
     \includegraphics{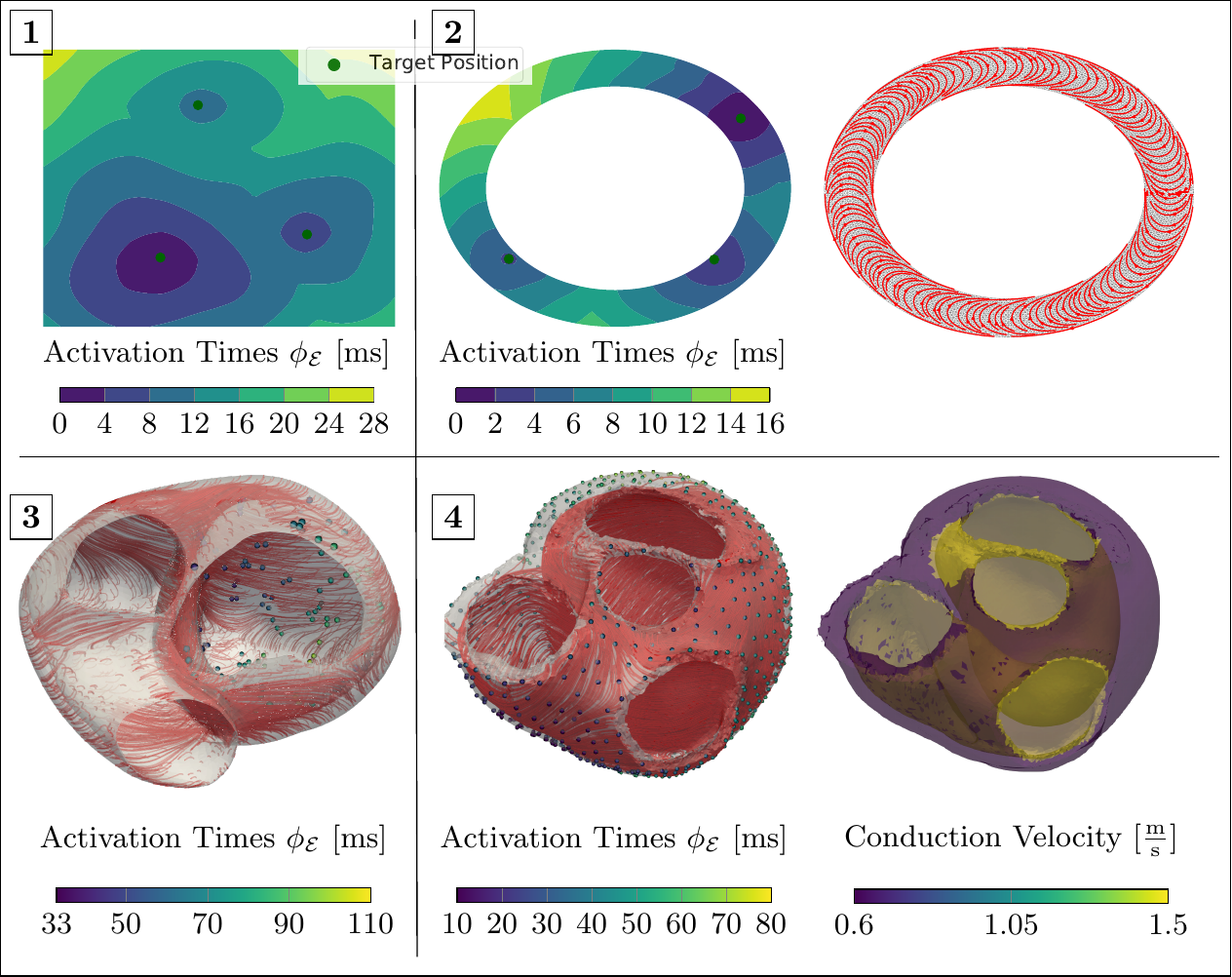}
}
     \caption{
     Activation times for all setups considered along with fiber orientations (if available).
     The isotropic conduction velocity of \protect\fbox{$\mathbf{1}$} is presented in \Cref{fig:geodesics} and exhibits no fiber orientation due to isotropy.
     Note that \protect\fbox{$\mathbf{3}$} was measured in-vivo and thus no ground truth is available.}
     \label{fig:exp_setups}
 \end{figure}

\subsection{Activation Time Optimization}\label{sec:act_times_exps}
In \Cref{fig:results_2d}, we present the results of \TheMethod for the 2D experiments: 
In the first iterations of the square example, the EASs are moved to the center of the domain to promote a good overall fit during optimization.
As the sites approach the center, fine details on the boundary can be fitted by minimizing the mismatch defining the optimal points.
The idealized LV model additionally requires a non-convex projection since the fiber alignment favors movements on the endocardial wall. 
The optimization still works for this case, even though the problem in~\eqref{eq:moreau_envelope} becomes non-convex.

\begin{figure}[htb]
    \centering
    \includegraphics{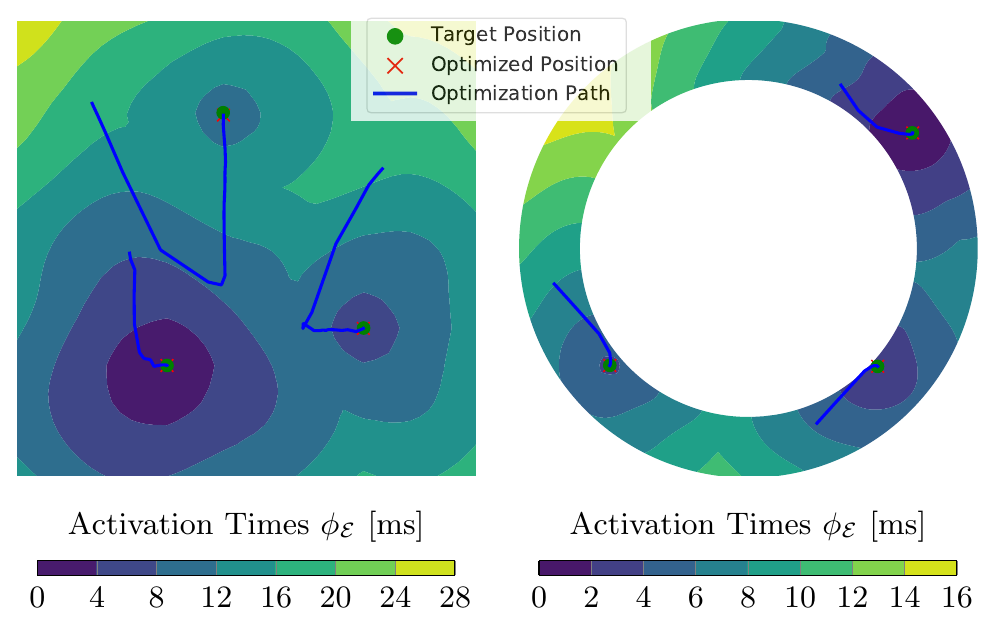}
    \caption{Results of \TheMethod for both 2D experiments. 
    We located the exact initiation sites with only a few iterations, both for the heterogeneous velocity case (left) and in the presence of non-convex projections for the idealized LV model (right).}
    \label{fig:results_2d}
\end{figure}

\begin{figure}[htb]

    \resizebox{\linewidth}{!}{
    \includegraphics{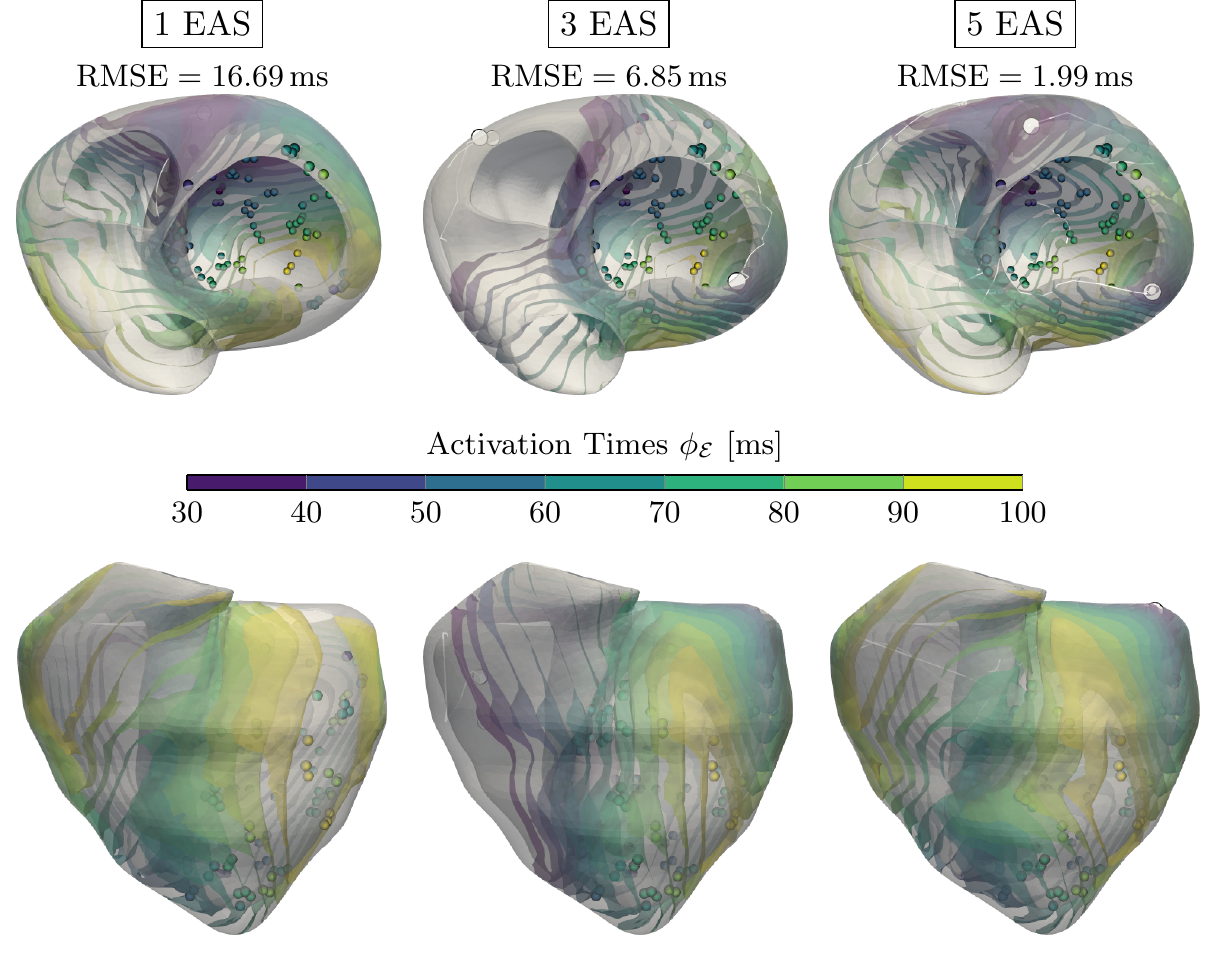}
    }
    \caption{Results for the CRT experiment with varying number~$N$ of EASs along with the RMSE (in ms) shown above each experiment. 
    The color-coded spheres indicate the observed activation times, while the white circles represent the optimized EAS positions. 
    The white trailing paths show the optimization path over the iterations. 
    Increasing $N$ lowers the overall RMSE, but may result in physiologically unlikely EAS (e.g.~top of the left ventricle for $N=5$).}
    \label{fig:results_3d_crt}
\end{figure}

\begin{figure}[htb]
    \resizebox{\linewidth}{!}{
    \includegraphics{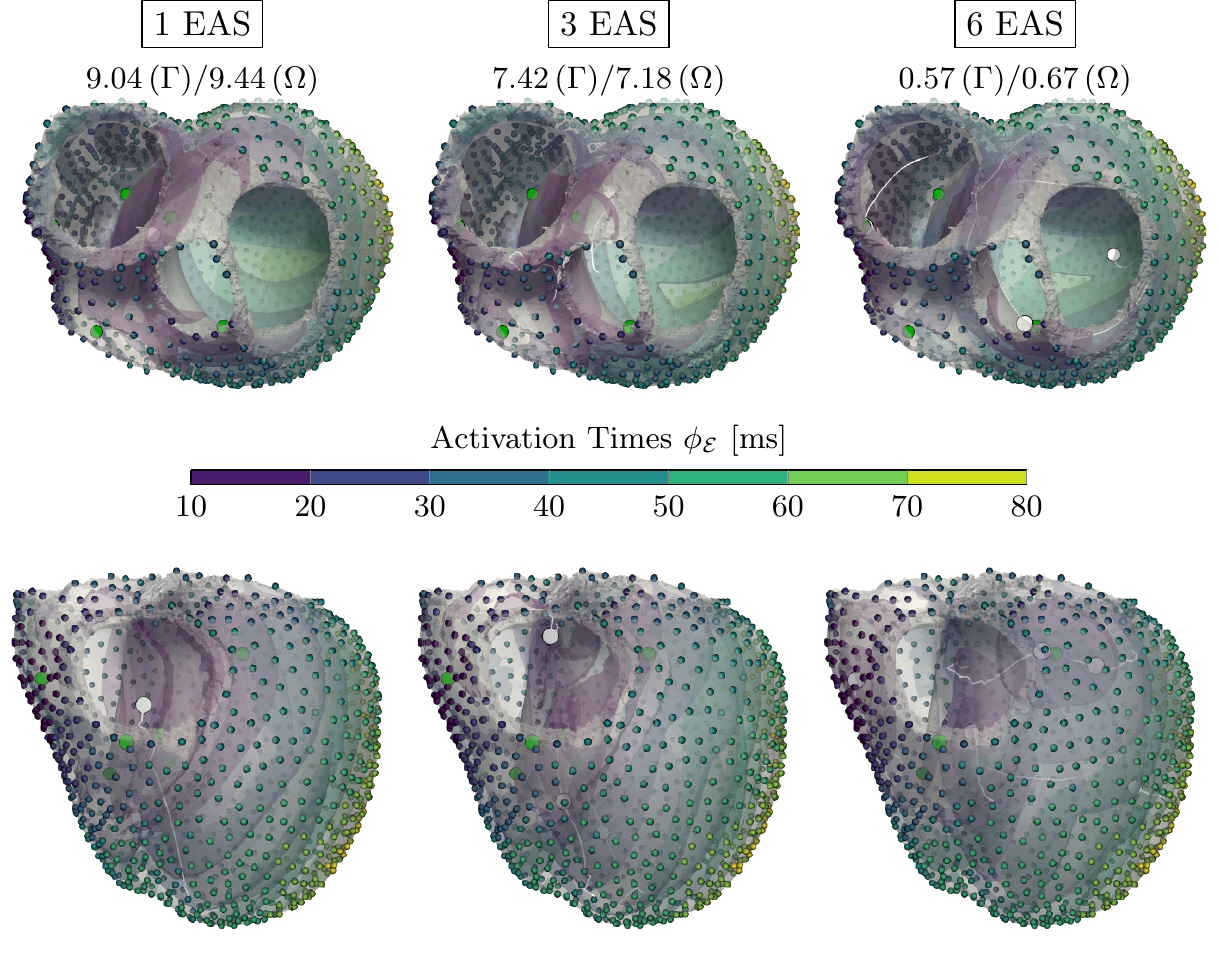}
    }
    \caption{Results for the trifascicular experiment along with the RMSE (in ms) shown above the experiments for both~$\Gamma$ and~$\Omega$.
    The color-coded spheres indicate the observed activation times.
    The white and green circles represent the optimized and target EAS, respectively.
    The overall RMSE activation error is very low if using the correct number of initiation sites ($N=6$), but we already obtain a good fit with fewer sites.
    }
    \label{fig:results_3d_trif}
\end{figure}

Next, we concentrate on 3D experiments in \Cref{fig:results_3d_crt,fig:results_3d_trif}, for which we alter the number of initiation points for both models. 
Even though we can not ensure that the activation of the clinically acquired CRT patient can be described by the eikonal model with the simple rule-based fiber orientation, the results on the CRT masurements provide an overall low root-mean-square error (RMSE) between modelled and measured activation times.
In the presence of a single EAS, the fit is (expectedly) sub-optimal since the activation requires a more complex activation pattern.
With three or more EASs, we get a much better fit, evenly distributed throughout the ventricles, but additional initiation sites are moved from the septum to the LV.
Note that there is no guarantee that the chosen rule-based fiber mapping can properly model the encountered ventricular activation.
In such a case, the additionally employed initiation sites are able to compensate possible modelling inaccuracies.
\label{rev:poorly_phrased}
We can achieve even better results by sucessively increasing the number of initiation sites, but this only reveals the nature of the ill-posed problem: 
By increasing the complexity of our model, we can more closely approximate the presented activation map (cf.~\Cref{prop:decreasing}). 

The trifascicular model has a higher resolution in comparison to the CRT model with an added fast conducting sub-endocardial layer, which is utilized in all longer geodesics from the measurements to the initiation sites. 
For this reason, it is important to properly project the geodesics in each iteration on the endocardium in a fast way.
For further details of the actual implementation we refer the reader to \Cref{sec:runtime}.
As a result, when using less EASs than in the ground truth we already achieve convincing numerical results, which is visualized in~\Cref{fig:results_3d_trif}.
If we incorporate 6~initiation sites, we get a very good fit, even though one of the activation sites is deactivated before convergence due to \eqref{eq:compatibilityCondition}.
The three septal points are jointly modelled by two EASs accounting for the deactivated point. 
Adding points beyond the given ones did not yield any improvement as they are deactivated by other points during the optimization (not shown).

\subsection{Topological Gradient}\label{sec:experiments_topo_grad}
We also tried to estimate the correct number of EASs by using topological gradients (see \Cref{sec:topo_grad_method}), where we analyzed all 2D setups and the CRT patient.
In this case, a splitting can only occur if the ratio~$\frac{\nu_{M,\varepsilon}}{\nu_{S,\varepsilon}}$ is below $10^{-1}$ (2D)/$2.5\cdot 10^{-1}$ (3D)
and the maximum Euclidean distance of the position of two consecutive iterates among all EASs is smaller than $10^{-2}h$ (with $h$ being the average element spacing).

The minimizer of~\eqref{eq:min_split} is chosen by evaluating $360$ (2D)/$5625$ (3D) directions, which are evenly distributed on the hypersphere.
We additionally ensure that the splitting direction is feasible (i.e.~it does not point outside the domain) by projecting the samples onto the mesh.
To avoid two coalescing EASs inside one element after a split, the points are moved apart by $2h$ from the original site.

In \Cref{fig:results_topo_grad_2d}, we collected the results for all 2D experiments using this method and plot the ratio $\frac{\nu_{M,\varepsilon}}{\nu_{S,\varepsilon}}$ over the iterations.
The first EAS is moved towards the center of the ground truth EASs, and subsequently several splits occur that closely match the ground truth sites. 
A similar behavior can be observed in the idealized LV model.

For the CRT patient in \Cref{fig:results_topo_grad_3d}, neither the ground truth EASs nor the fiber distribution and velocities in~$\Omega$ are known.
In total, the algorithm introduced $8$~splits (i.e.~$N=9$), of which $4$ are deactivated during optimization since they violated \eqref{eq:compatibilityCondition}.
Only those final EASs are shown in the right plot of \Cref{fig:results_topo_grad_3d}.
Moreover, we can see that three main clusters are identified, where one initiation cluster is located at the upper part of the anterior septum.
The optimization in this region is further complicated by the very thin wall of the 3D mesh, which likely causes the high number of splits.
We highlight that constant (in time) split ratios are caused by temporarily deactivated EASs violating~\eqref{eq:compatibilityCondition}.
To conclude, we get a tremendous fit with the presented measurement points despite the aforementioned model assumptions.
Moreover, the topological gradient could be successfully applied to all 2D models leading to the correct estimate for the number of EASs and also matching the correct sites.
The corresponding results in 3D provide a very low overall RMSE on the measurements.

\begin{figure}[htb]
    \centering
    \includegraphics{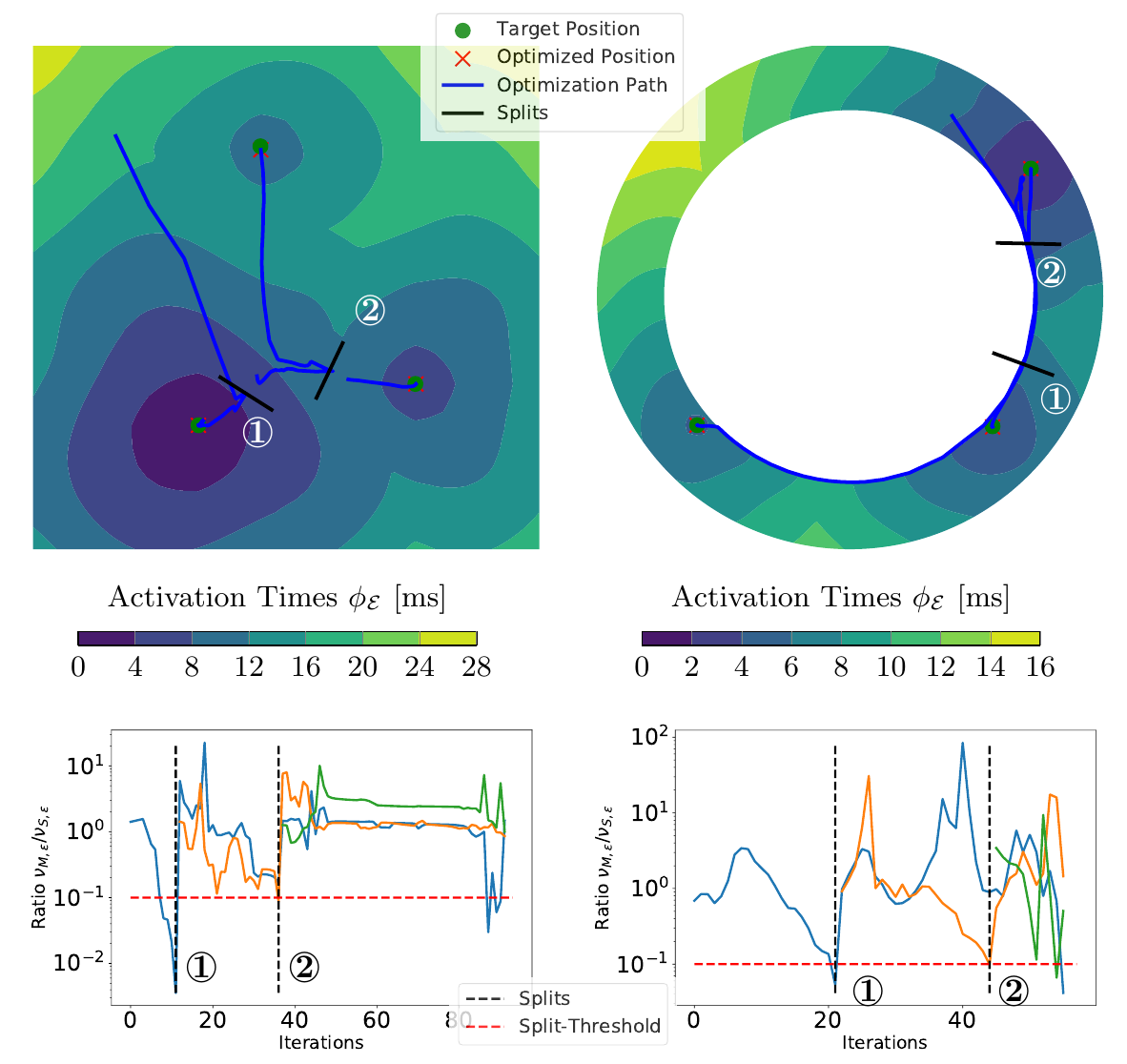}
    \caption{Results of the 2D experiments for the topological gradient.
    Top row: optimization paths starting with a single EAS.
    Bottom row: plots of $\frac{\nu_{M,\varepsilon}}{\nu_{S,\varepsilon}}$ for each EAS depending on the iterations, where an EAS is split if this ratio is below the dotted red line.
    The location (top) and iteration (bottom) of the splits are marked by \mytextcirc{1} and \mytextcirc{2}.
    Note that an EAS only splits if all parameters have converged (see \Cref{sec:experiments_topo_grad}).}
    \label{fig:results_topo_grad_2d}
\end{figure}

\begin{figure}[htb]
    \resizebox*{\linewidth}{!}{
    \includegraphics{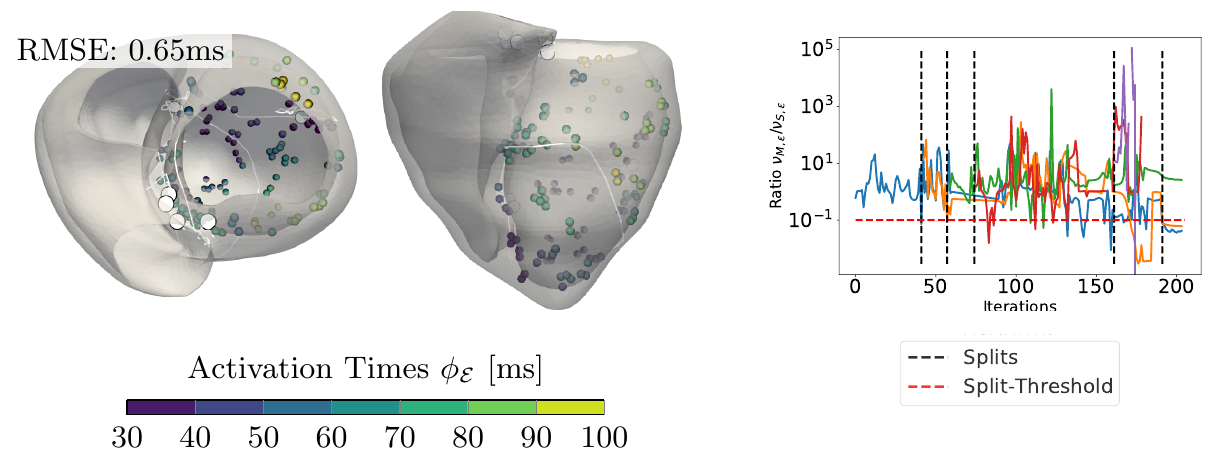}
    }
    \caption{Results for the topological gradient extension on the CRT experiment with a visualization analogous to \Cref{fig:results_topo_grad_2d}.}
    \label{fig:results_topo_grad_3d}
\end{figure}

\subsection{ECG}\label{sec:experiments_ecg}
In what follows, we present numerical results for the ECG optimization for both 2D experiments as well as the trifascicular model in a simplified fashion as a proof-of-concept.
The ECG requires an additional full torso domain~$\Omega_T$ and the computation of the lead fields.
For all experiments in this paper, we embedded all three in-silico experiments (i.e.~\fbox{$\mathbf{1}$},\fbox{$\mathbf{2}$},\fbox{$\mathbf{3}$}) into a non-equilateral cube-torso without any additional organs and an overall torso conductivity of \SI{0.2}{\milli\siemens\per\mm}.
The size of the cube-torso is proportional to the bounding-box of~$\Omega$.
The computed lead fields are shown in \Cref{fig:ecg_setups}. 
In all cases, we generated a noiseless target ECG from the reference model setup with parameters and initiation sites already presented in \Cref{sec:experiments}.
We optimize our model with random initialization with respect to this target ECG.
Note that we do not focus on the generated ECGs' absolute potentials, since this heavily depends on the actual torso setup.
Instead, we rather focus on the overall morphology of the ECGs.

\label{rev:R2green}
To compute the lead fields in \eqref{eq:lead_field}, our cube-torso is sampled using a structured regular grid of $100^d$ equidistant points, and the problem is solved with a finite difference scheme, which is sufficiently accurate since the lead field is evaluated far away from the singularity~\cite{potse2018}.
The lead fields are computed prior to the optimization since they remain constant.
The ECG signals for the 3D models are mean-filtered with a small kernel of size $\approx \SI{2}{\ms}$ to improve accuracy.

\begin{figure}[htb]
    \centering
    \includegraphics{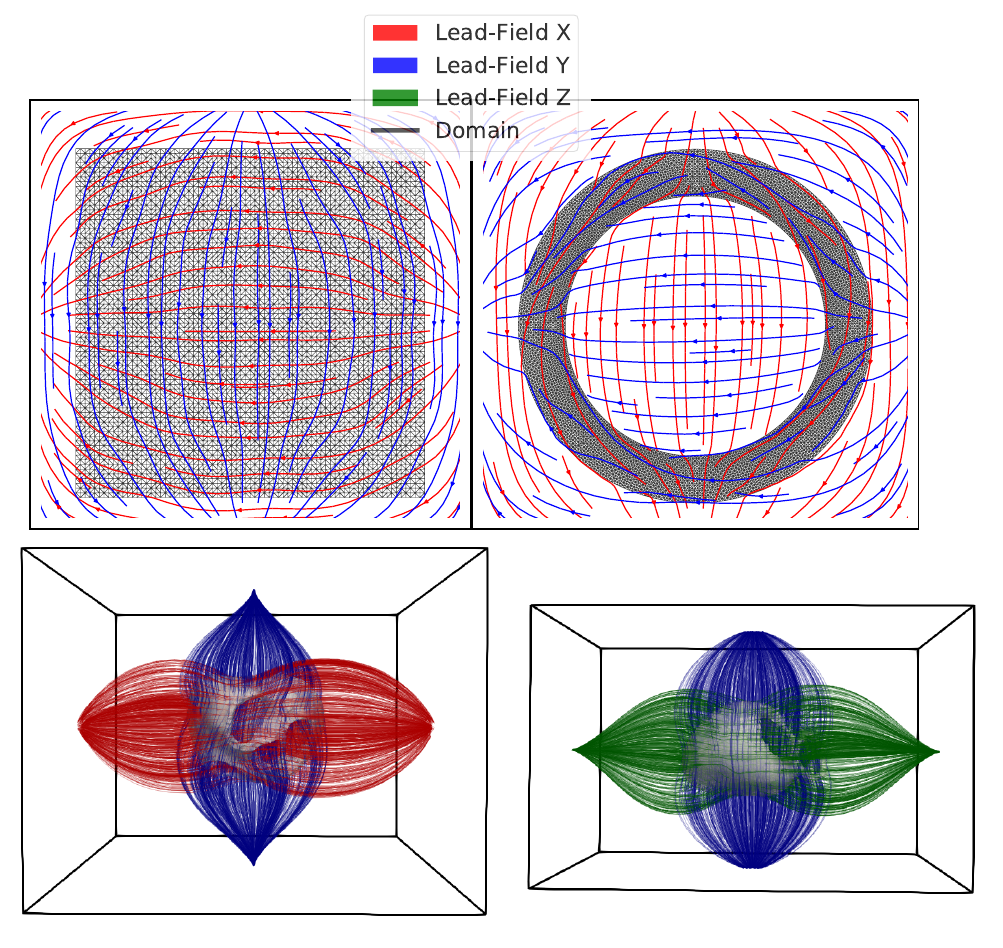}
    \caption{Setup for the ECG experiments showing the torso domain $\Omega_T$.
    The heart domain $\Omega$ is indicated by black lines for the 2D experiments and gray silhouettes for 3D. 
    The streamlines visualize the lead fields.
    Note that the lead field for axis Z (green) is only present in the 3D experiments.}
    \label{fig:ecg_setups}
\end{figure}

Optimization solely based on the ECG is frequently very challenging.
However, with a proper initialization $(\vx_i,t_i)$, good fits for the ECGs can be computed.
\Cref{fig:results_ecg} shows the optimization paths, as well as initial, target and optimized ECGs using the modified \TheMethod algorithm presented in \Cref{sec:ecg_method} for the 2D examples, which are computed in approximately $2.5$~hours each.
The two potentials are a result of the two axis-aligned lead-fields (see \Cref{fig:ecg_setups}). 

In the numerical experiments, it turned out that that the overall step size~$\beta_s$ has to be chosen smaller compared to the activation timing problem.
The morphology of the initial ECG and the optimized ECG differ by a large margin, making the fitting non-trivial. 
As a result, in both the square domain and the idealized LV experiment we are able to closely match the actual sites from which the target ECG was generated (\Cref{fig:results_ecg}, second row).

\begin{figure}[htb]
    \centering
    \includegraphics{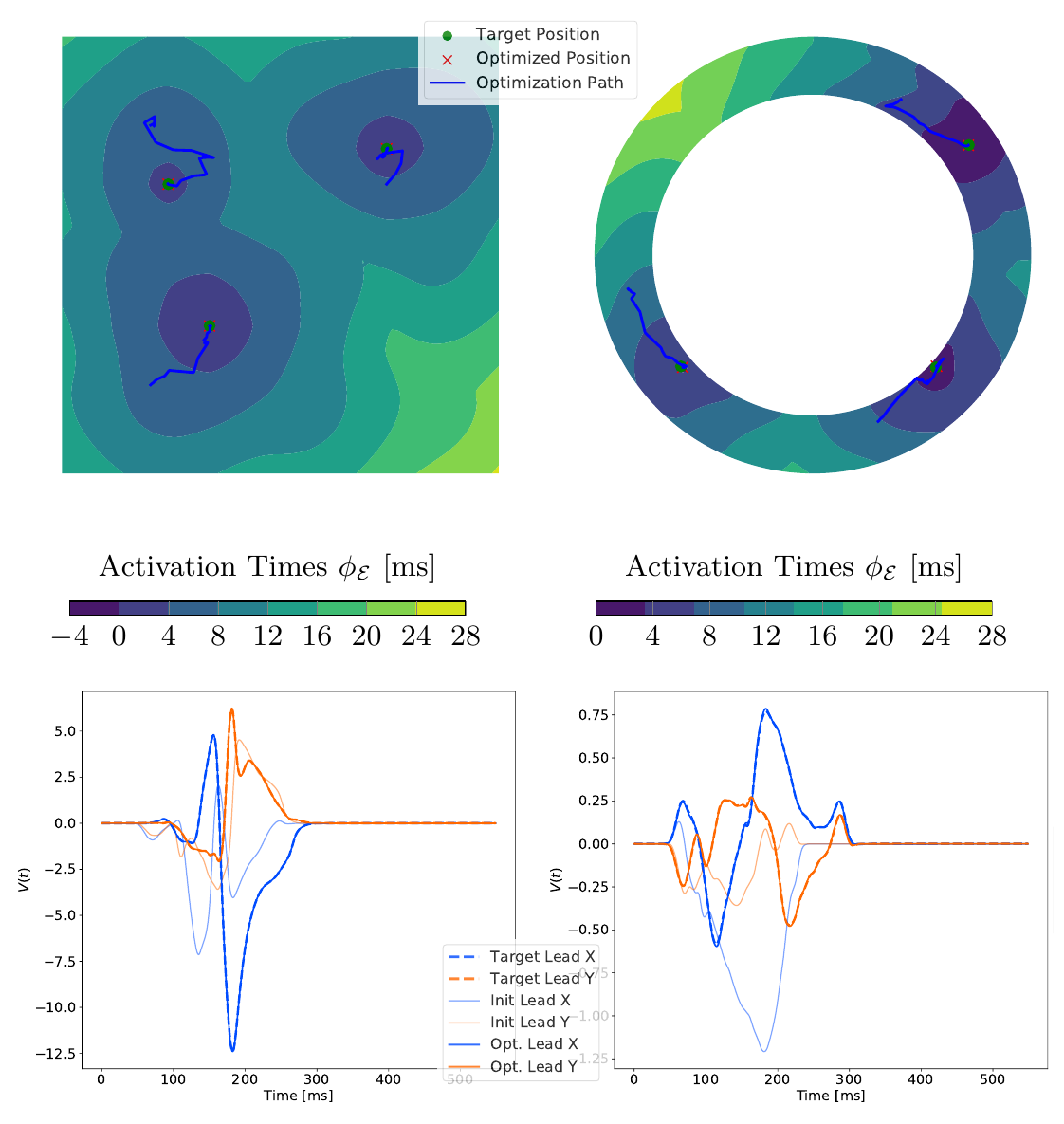}
    \caption{Results of the 2D ECG optimization. 
    Top row: temporal change of the positions of the EASs along with the ground truth.
    Bottom row: initial, final and target ECG for fitting.}
    \label{fig:results_ecg}
\end{figure}

\begin{figure}[htb]
    \centering
    \includegraphics{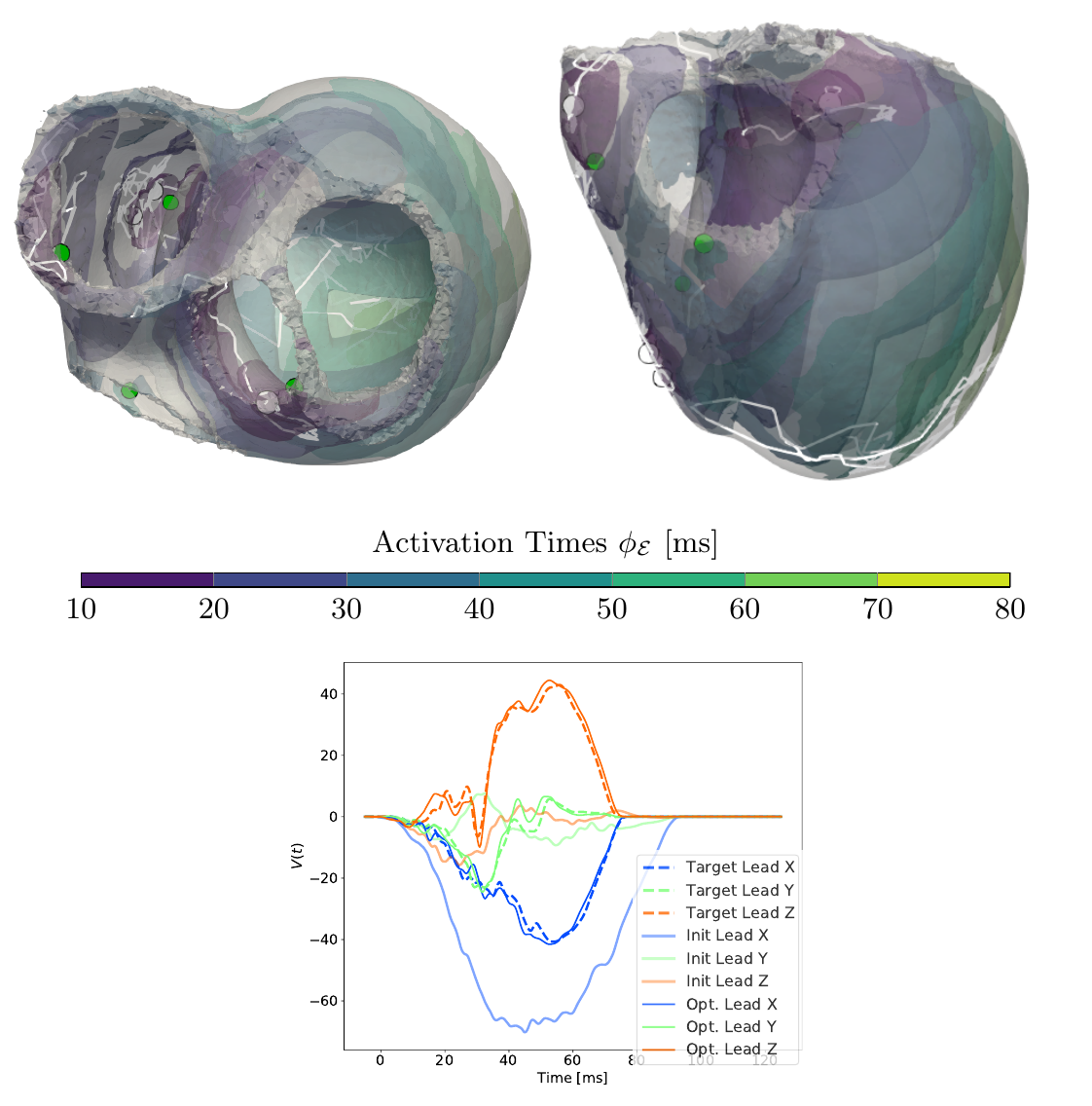}
    \caption{Results of the ECG optimization on the trifascicular model. 
    Top row: optimized positions of EASs (white circles) along with temporal changes over the iterations (white lines).
    The green circles represent the target position from which the ground truth was generated.
    Bottom row: initial, final and target ECG.}
    \label{fig:results_triF_ecg}
\end{figure}

The trifascicular model in \Cref{fig:results_triF_ecg} is computationally demanding since in each iteration step a computation of all geodesics is required, i.e.~we need to solve $\approx 10^{5}$ ODEs per iteration (for further details we refer to \Cref{sec:runtime}).
As each initial EAS is randomly chosen, the initial ECG significantly differs from the target.
Note that the 3D cube torso exhibits three axis-aligned leads.
Since lead-X and lead-Z have the most prominent peaks, they have the largest effect on the resulting $L^2$-error.
After the optimization, these peaks were fitted by the algorithm by shifting most of the initiation sites to the LV and one to the anterior wall and septal region.
The added difficulty with an activation featuring that many EASs is also apparent from the computed paths (white lines) which strongly vary during optimization.
After termination, $4$ of $6$ sites are close to the ground truth sites defining the target ECG.

\section{Discussion}\label{sec:discussion}
In the previous section, we have experimentally demonstrated the broad applicability of the proposed \TheMethod method for a variety of problems.
Despite the convincing results there are still some issues related to our approach and alternative approaches, which will be addressed in future work.

\subsection{Eikonal Equation}
\label{sec:discussion_eikonal}
In this work, we rely on the anisotropic eikonal equation, but other versions thereof are also applicable.
More specifically, several eikonal frameworks to model physical and medical processes have been proposed over the last three decades, which
can be derived from either the monodomain or the bidomain equation using a perturbation argument~\cite{colli2014}.
The most common equation inferred from a first-order approximation of the monodomain equation is the anisotropic eikonal equation~\eqref{eq:eikonal}.
The eikonal model originating from the bidomain model is slightly different and is based on a Finsler-type metric~\cite{ambrosio2000asymptotic}. 

Second-order approximations lead to the curvature-eikonal, diffusion-eikonal and viscous-eikonal equations.
In the curvature-eikonal model~\cite{keener1991eikonal}, the front velocity is corrected by the curvature of the front in the metric induced by the conductivity tensor.
In contrast, in the diffusion-eikonal equation~\cite{franzone1990wavefront} a diffusion term is added to the right-hand side of~\eqref{eq:eikonal}.
Finally, in the viscous-eikonal model~\cite{kunisch_inverse_2019}, a squared eikonal equation is considered, which is corrected by a diffusion term.

Higher-order approximations have also been proposed, but are rarely used in practice~\cite{dierckx2011accurate}.
The effect of higher-order terms is more pronounced in front collisions, at the boundary of the domain and in narrow channels, e.g., in scarred tissue.
In practice, however, deviations from the classical eikonal model are minimal and is therefore widely accepted for personalization of cardiac models.

The distinction between these models is however important from the point of view of the EASs and is often dictated by the numerical method rather than the physiology.
In the standard anisotropic eikonal model~\eqref{eq:eikonal}, EASs can be single points, whereas in the curvature-eikonal equation EASs are required to have a strictly positive Lebesgue measure.
Note that the conduction velocity of a spherical front with small radius is significantly slower in the presence of higher-order correction terms.

\subsection{Runtime}\label{sec:runtime}
The majority of the computational time is spent for solving the geodesics in \eqref{eq:ode_eikonal_geodesics}, performed in parallel on the GPU.
We highlight that the number of geodesics is proportional to the size of~$\Gamma$ in the original version (\Cref{alg:geasi_phi}) and proportional to~$\Omega$ in the modified version (\Cref{sec:ecg_method}).
The computation of all geodesics in both cases is performed in parallel on a GPU and therefore scales well with the mesh size. 
The bulk of computational time inside the ODE solver is spent on the projection of each ODE solution back onto the mesh and nearest neighbor computation. 
For the nearest neighbor computation, we implemented a custom KD-Tree implementation (publicly available on GitHub\footnote{\url{https://github.com/thomgrand/tf_kdtree}}). 
For the projection operator, we extract the surface of the mesh, prior to the computation using the truncated signed distance function from VTK\footnote{\url{https://vtk.org/}}.
The K-nearest neighbor elements of the current positions of the geodesics are then queried to calculate the analytical projection onto all reference elements.
The projection to all nearest neighbors is the minimum distance projection onto the mesh~$\Omega$.
\label{rev:R5geo}
Note that the time of each individual ODE solution in~\eqref{eq:ode_eikonal_geodesics} depends on the length of the associated geodesic.
Since adding new points can only shorten geodesic lengths (compare~\eqref{eq:distance} and \Cref{sec:topo_grad_method}), more EASs will result in faster convergence and reduced computation time.

Solving the eikonal equation in \eqref{eq:eikonal} as well as the Gauss--Newton optimization in \eqref{eq:opt_problem} only requires a minor portion of the computational time.
As already mentioned, the activation time optimization is much faster compared to the remaining computations.
In total, the experiments were finished within about 100 iterations only taking approximately 30 minutes and 90 minutes for 2D and 3D experiments, respectively.
The experiments for the topological gradient behaves similarly regarding computational time.
In contrast, the 3D optimization in the ECG problem requires approximately 12 hours.

To further decrease runtime, several approaches are possible: 
A custom GPU implementation to solve \eqref{eq:ode_eikonal_geodesics} along with the projection could significantly speed-up the optimization. 
Additionally, we often witnessed a collapse of many geodesic paths, especially in the trifascicular model, making subsequent computations redundant. 
An adaptive sampling from the measurement domain $\Gamma$ combined with a proper upsampling technique could increase performance at the cost of precision.

 \begin{figure}[htb]
     \centering
     \includegraphics[width=.5\textwidth]{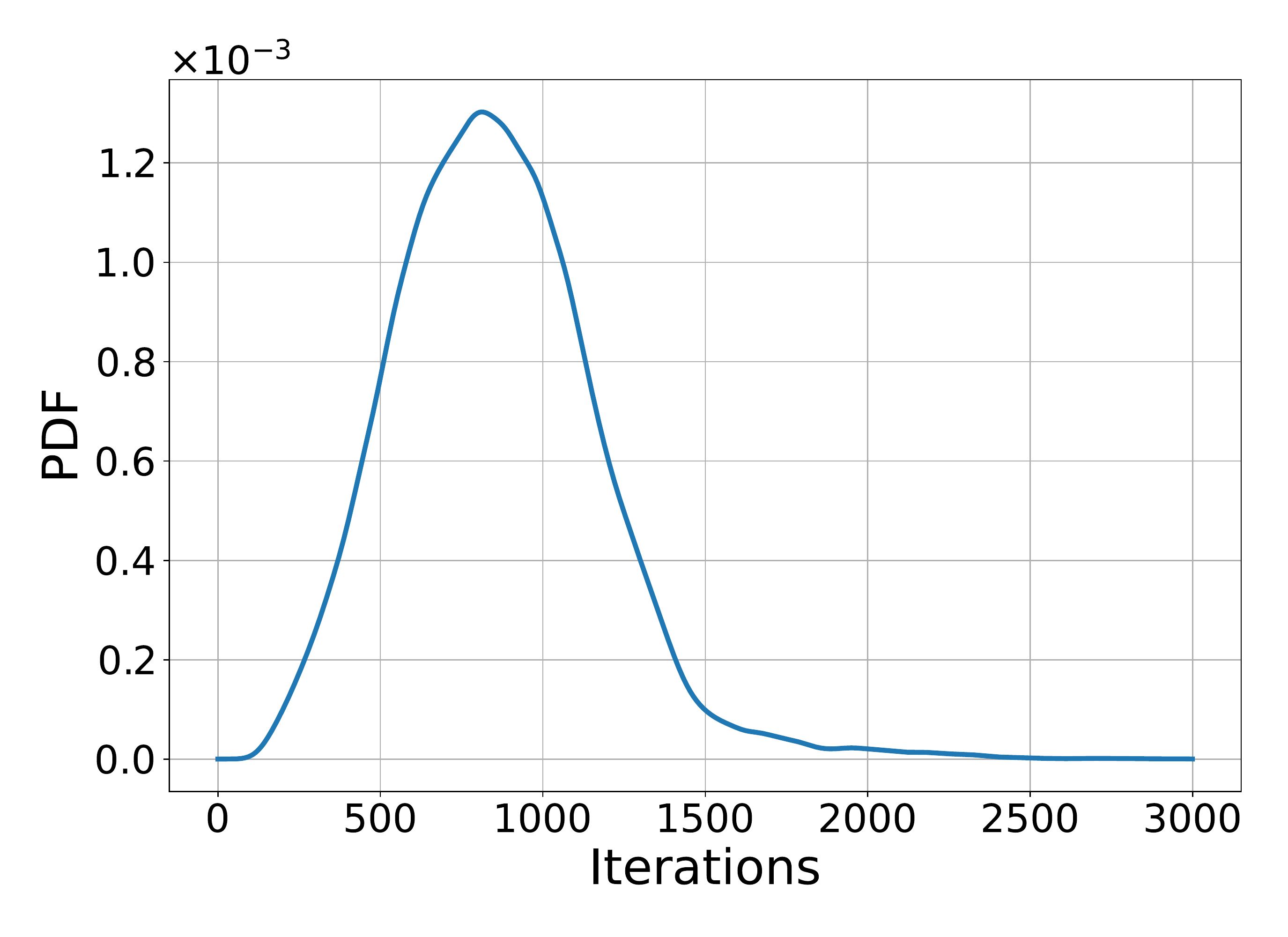}
     \caption{Convergence of the geodesic ODE in \eqref{eq:ode_eikonal_geodesics} for the trifascicular model over the iterations with a single EAS in the septum.
     The majority of the geodesics converge before 2000~iterations.
     }
     \label{fig:geodesics_convergence}
 \end{figure}
 
To improve performance for the 3D ECG optimization, we analyzed the convergence of the ODEs.
\Cref{fig:geodesics_convergence} shows a probability density function (PDF) of convergence of the geodesics $\gamma$ over the number of required iterations using the trifascicular model with a single initiation site in the septum.
Convergence in this case is defined as the first time two subsequent ODE iterations of~\eqref{eq:ode_eikonal_geodesics} have a change of less than $10^{-10}$, i.e.~$\norm{\gamma(t_{k+1})-\gamma(t_{k})}<10^{-10}$. 
We see that many of the computed geodesics converge very quickly, while points with a high geodesic distance need significantly more iterations before convergence. 
Our vectorized/parallel implementation to solve \eqref{eq:ode_eikonal_geodesics} exploits this fact to only include non-converged geodesics.

\subsection{ECG}
\label{sec:discussion_ecg}
The ECG results demonstrated that \TheMethod can be used to fit a given ECG. 
However, one main problem is getting stuck in local minima. 
While the $L^2$-error is relatively low in these minima, the morphology of the optimized and the target ECG differ a lot. 
One of the main reasons for this problem could result from the usage of the $L^2$-error, which is not robust to transformations of the time series, such as time shifts.
Better error measures for this type of optimization include dynamic time warping \cite{sakoe_dynamic_1978} and the Wasserstein distance~\cite{yang_application_2017}.
Finally, different optimization algorithms could further help to overcome this issue.

\section{Conclusion and Future Work}
\label{sec:conclusion}
This paper introduced the novel \TheMethod method to find the optimal source points of an eikonal model with a special focus on electrophysiological examples. 
We showed that \TheMethod can model complex eikonal activations, either from the activation times directly, or by fitting a given ECG. 
For our model examples, we were able to identify most of the ground truth EASs along with times, and in the case of the topological gradient also the number.
We were even able to model CRT measured data with only a few source points.


So far, we only assumed fixed conduction velocities and fiber distributions for all cases.
In future studies we intend to estimate these parameters using the same procedure with only minor necessary modifications to \eqref{eq:opt_problem}.
We note that past studies \cite{grandits_inverse_2020,Grandits2020PIEMAP} have already shown that for the optimization of conductivities, additional regularization is crucial to decrease model complexity.

\label{rev:R1rd}
\TheMethod is inherently connected to the anisotropic eikonal equation through the Hamilton--Jacobi formalism. 
Thus, a possible extension of \TheMethod to reaction-diffusion models, possibly with non-local diffusion terms~\cite{bueno2014fractional} and more complex boundary effects~\cite{bishop2011,vergara2016coupled}, is not trivial and probably requires a hybrid reaction-diffusion-eikonal approach~\cite{neic_efficient_2017}.

All extensions of \TheMethod such as topological gradient and ECG optimization hold much promise for future applications in clinical real-world examples.
\TheMethod faces several computational hurdles, many of which we already tackled in this study.
We hope to further improve and expand \TheMethod---both methodologically and computationally---to enlarge the applicability to a wide-range of problems.
\label{rev:R1appl}
Several pathological scenarios nicely fit in the \TheMethod framework, and the proposed method
could potentially greatly improve the identification of the site of origin of premature ventricular contractions and monomorphic ventricular tachycardia~\cite{calkins1993relation}.
\TheMethod could also be applied to improve planning of therapeutic interventions
such as cardiac rhythm management with optimal placement and number of pacing leads~\cite{auricchio2003clinical}.
Therefore, we believe that \TheMethod has the potential to significantly advance and improve personalized health care in the future.

\paragraph{Acknowledgements}
This work was financially supported by the Theo Rossi di Montelera Foundation,
the Metis Foundation Sergio Mantegazza,
the Fidinam Foundation,
the Horten Foundation and the CSCS—Swiss National Supercomputing Centre production grant s778.
This research was supported by the grants F3210-N18 and I2760-B30
from the Austrian Science Fund (FWF) and BioTechMed Graz flagship award 
``ILearnHeart'', as well as ERC Starting grant HOMOVIS, No.\ 640156.
A.~Effland was funded by the German Research Foundation under Germany's Excellence Strategy -- EXC-2047/1 -- 390685813 and -- EXC2151 -- 390873048.

\bibliographystyle{ieeetr}
\bibliography{geasi,appl}

\end{document}